\documentclass[11pt]{article}

    \usepackage{url}
    \usepackage{verbatim}
    \usepackage[titletoc]{appendix}
    \usepackage{bm}
    \usepackage{amsthm}
    \usepackage{amsmath}
    \usepackage{mathrsfs}
	\usepackage{amsfonts}
    \usepackage{nicefrac}
    \usepackage{longtable}
    \usepackage{color}
    \usepackage{graphicx, amssymb,graphics}
    \usepackage{algorithmicx,algorithm}
    \usepackage{algpseudocode}
    \usepackage{indentfirst}

    \usepackage{paralist}

    \usepackage{graphicx}
    \usepackage{subcaption}   
    \usepackage{multirow}

\usepackage{xcolor} 
\definecolor{tcr}{named}{red}        
\definecolor{tcb}{named}{blue}
\newcommand{\tcr}[1]{\textcolor{tcr}{#1}}  

\usepackage[colorlinks, linkcolor=red, anchorcolor=red, citecolor=red]{hyperref}

\newtheorem{thm}{Theorem}[section]

\newtheorem{lem}{Lemma}[section]

\newtheorem{proposition}{Proposition}
\newtheorem{remark}{Remark}

	\newcommand\be {\begin{equation}}
	\newcommand\ee {\end{equation}}

\numberwithin{equation}{section}

\graphicspath{{..}}

\title{Local Relaxation Fast Poisson Methods on Hierarchical Meshes} 
\date{\today}

\begin{document}
	
\author{
Zhenli Xu \thanks{School of Mathematical Sciences, MOE-LSC, and CMA-Shanghai, Shanghai Jiao Tong University, Shanghai 200240, China. (xuzl@sjtu.edu.cn)},~
\and
Qian Yin \thanks{Department of Applied Mathematics, the Hong Kong Polytechnic University, Hong Kong, China. (qqian.yin@polyu.edu.hk)},~
\and
Hongyu Zhou \thanks{Corresponding author. School of Mathematical Sciences, MOE-LSC, and CMA-Shanghai, Shanghai Jiao Tong University, Shanghai 200240, China. (zhouhongyu@sjtu.edu.cn)}
}

\maketitle

\begin{abstract}

The local relaxation algorithm is promising for fast solution of Poisson's equations, which computes the electric field distribution in a stepwise manner via local curl-free updates while strictly enforcing Gauss’s law. We propose a novel hierarchical local relaxation (HLR) method for speeding up the convergence of curl-free iterations. The local algorithm reformulates the Poisson's equation into the electric-field form and sweeps each cell to minimize the associate electric energy, avoiding the solution of linear systems. The updates with hierarchical meshes significantly accelerate the slow convergence of low-frequency components of the residual in the local curl-free update process. Convergence analysis is performed to obtain the convergence of the hierarchical relaxation approaches. Numerical results show that the HLR methods have the nice properties in accuracy and efficiency and the hierarchical construction leads to an overall computational complexity of $\mathcal{O}(N\log N)$ with respect to the number of grid points. Particularly, the applications in solving the Poisson--Boltzmann and Poisson--Nernst--Planck equations demonstrate the attractive performance for problems which frequently solve the Poisson's equations. 
 
\bigskip


\noindent
\end{abstract}

\section{Introduction}

The Poisson’s equation is a fundamental model describing the relation between source distributions and potential fields and arises ubiquitously in many fields \cite{clapham2007calcium,hockney2021computer,maury2001fat,van2006charge}. In molecular dynamics (MD) simulations, accurate evaluation of long-range electrostatic forces relies on the repeated solution of the Poisson's  equation or its equivalent formulations \cite{honig1995classical,roux2002theoretical,waisman1972mean}. Besides, it is coupled with the Nernst–Planck or Vlasov equations to describe ionic transport and the self-consistent evolution of charged particles under electric fields \cite{kim2005fully,zheng2011second}. In all cases, efficient and scalable solvers are essential for efficient simulations.  Repeatedly solving the Poisson's equation is a major computational for simulating large or spatially heterogeneous systems. The finite element methods (FEMs) \cite{dhatt2012finite,sorokina2022interpolated,zienkiewicz1977finite} and boundary element methods (BEMs) \cite{cheng2005heritage,desiderio2022cvem,hall1994boundary} are two classes of traditional methods for the Poisson's equation, but they are considered less efficient for applications of time-dependent problems. The FEM requires repeated matrix assembly when the dielectric permittivity evolves in time. The BEM relies on analytical fundamental solutions, which are generally unavailable for problems with spatially varying dielectric permittivity.  Local algorithms provide an efficient alternative by treating the electric field as the fundamental variable and enforcing Gauss’s law as a constraint, rather than directly solving for the electrostatic potential \cite{M:JCP:2002,rottler2004continuum}. Their local formulation accommodates time-dependent dielectric permittivity without extra cost compared to the constant permittivity case.

The local algorithm was first introduced for systems of discrete charges, for which Maggs and his collaborators \cite{M:JCP:2002,maggs2002local,rottler2004continuum} proposed to use curl-free relaxation for performing local Monte Carlo simulations of Coulomb systems. This idea was extended to MD methods \cite{fahrenberger2014simulation,pasichnyk2004coulomb} with converting globally coupled Coulomb interactions into locally updated electromagnetic fields, enforcing Gauss’s law and quasi-static constraints while preserving thermodynamic properties. Such local approaches enable linear scaling and are naturally parallelizable, while maintaining physically consistent particle–field dynamics. Subsequently, the same principle was applied to the solution of the Poisson–Boltzmann (PB) equation, by reformulating the PB model as a variational problem that minimizes a free energy functional and solving it using a Maggs-type algorithm \cite{BSD:PRE:2009,ZWL:PRE:2011}. Recently, Qiao {\it et al.} \cite{qiao2023maxwell,qiao2023structure} further generalized local algorithms to the Poisson–Nernst–Planck (PNP) model, developing the Maxwell–Amp\`ere Nernst–Planck (MANP) framework to describe charge transport. Besides, local algorithms provide an efficient and physically consistent strategy for particle-in-cell (PIC) simulations \cite{qiao2025intrinsic}, allowing electric fields and particle positions to be updated locally while strictly preserving Gauss’s law and supporting large-scale, parallel computations.
 
In existing literature, electric field is updated locally within each discretized cell. However, such relaxation update is less efficient in eliminating low-frequency components of the curl residual. Motivated by multigrid methods for addressing the slow convergence of Gauss–Seidel iterations for low-frequency errors in linear systems \cite{hemker1981introduction,mccormick1987multigrid,milaszewicz1987improving}, we propose an efficient  hierarchical local relaxation (HLR) method for solving Poisson’s equation. The HLR method decomposes the computational domain into a hierarchy of grids with different resolutions and performs curl-free relaxation at each level. Relaxation on fine grids efficiently damps high-frequency errors, while coarse-grid corrections effectively reduce low-frequency errors, leading to accelerated overall convergence. We show that this idea can be naturally extended to other algorithms based on the Maggs-type local formulation, including solvers for the PB equation and the MANP system. Moreover, by employing the method of image charges \cite{mesa1996image}, we  generalize the local algorithms to handle non-periodic boundary conditions.

Compared with traditional fast Poisson solvers such as FFT-based \cite{duhamel1990fast,feng2021fft,nussbaumer1982fast} and AMG-based \cite{ruge1987algebraic,stuben2001review} methods, the HLR method benefits from its local formulation, which naturally accommodates spatially and temporally varying dielectric permittivity. Unlike AMG-based approaches that require discretization into large sparse linear systems and repeated stiffness matrix assembly for time-dependent media, the HLR method avoids global algebraic formulation and matrix assembly altogether. Furthermore, by incorporating a hierarchical structure into the local curl-free relaxation, the HLR method efficiently removes low-frequency curl residuals, significantly accelerating convergence and improving overall computational efficiency. Due to these advantages, the HLR is promising for broad applications which require a fast Poisson solver.

The rest of this paper is organized as follows. In Section~\ref{remodeling}, we reformulate Poisson’s equation and show that solving the Poisson's equation is equivalent to computing the electric field. Section~\ref{la} introduces a class of local algorithms for the Poisson's equation and presents the HLR method. In Section~\ref{cn}, we provide a rigorous theoretical analysis of the proposed HLR method. Section~\ref{NumResults} presents numerical experiments that demonstrate the accuracy and efficiency of the proposed algorithm. Finally, conclusions are drawn in Section~\ref{s:Con}.

\section{Electric-field formulation of the Poisson's equation}\label{remodeling}

Consider the Poisson's equation of the form
\begin{equation}\label{poi}
    -\nabla\cdot(\varepsilon(\bm x)\nabla\phi)=\rho(\bm{x}),
\end{equation}
where $\phi$ is a scalar field, $\varepsilon$ is a spatially varying coefficient, and $\rho$ is a given source term. In charged systems, the Poisson's equation describes how the electrostatic potential is generated by a given charge distribution and mediated by the dielectric response of the medium with $\phi$ corresponding to the electric potential, $\varepsilon$ being the dielectric permittivity, and $\rho$ describing the charge density.
Accordingly, the electric field is given by $\bm E=-\nabla \phi$. With $\bm E$, the Poisson's equation can be written as,
\begin{subequations} \label{elctric field}
\begin{align}
    \text{Gauss's law: } &\nabla \cdot \left(\varepsilon\bm E\right) = \rho,\label{gaussn} \\
   \text{Curl free: }&\nabla \times \bm E = 0.\label{curlfreen}
\end{align}
\end{subequations}
For a bounded simply connected domain, the equivalence between \eqref{poi} and \eqref{elctric field} can be constructed by Helmholtz decomposition \cite{kirsch2015mathematical,monk2003finite}. 

The electrostatic energy of a charged system can be expressed as a functional of the electric field,
\begin{equation}\label{energyE}
\mathcal{F}_{\text{pot}}[\bm E] := \int_{\Omega} \frac{\varepsilon|\bm E|^2}{2} d\bm x.
\end{equation}

In electrostatics, the dielectric permittivity 
$\varepsilon$ is required to be positive to ensure a physically admissible and thermodynamically stable continuous medium, for which the electrostatic energy density remains positive definite. Proposition~\ref{p1} states the convex property of electrostatic energy, which enables the electrostatic field to be computed by solving a constrained energy functional minimization problem.

\begin{proposition}\label{p1}
The electrostatic energy 
 functional $\mathcal{F}_{\text{pot}}: L^{2}(\Omega,\mathbb{R}^n)\rightarrow\mathbb{R}$ is convex. Specifically, For any $\bm E_1,\bm E_2\in L^{2}(\Omega,\mathbb{R}^n)$, and $\forall \theta\in(0,1)$ we have
 \begin{equation}\label{convex}
     \mathcal{F}_{\text{pot}}(\theta\bm E_1+(1-\theta)\bm E_2)\leq \theta\mathcal{F}_{\text{pot}}(\bm E_1)+(1-\theta) \mathcal{F}_{\text{pot}}(\bm E_2).
 \end{equation}
The equality holds if and only if $\bm E_1=\bm E_2$.
\end{proposition}

\begin{proof}
Since $\varepsilon>0$, \eqref{convex} is equivalent to the following inequality,
\begin{equation}\label{eql}
    |\theta\bm E_1+(1-\theta)\bm E_2|^2\leq\theta|\bm E_1|^2+(1-\theta)|\bm E_2|^2.
\end{equation}
It remains to show that~\eqref{eql} holds,
\begin{align*}
    |\theta\bm E_1+(1-\theta)\bm E_2|^2&\leq(|\theta\bm E_1|+|(1-\theta)\bm E_2|)^2\\
    &=\theta^2|\,\bm E_1|^2+(1-\theta)^2|\bm E_2|^2+2\theta(1-\theta)|\bm E_1||\bm E_2|\\
    &\leq \theta^2|\bm E_1|^2+(1-\theta)^2|\bm E_2|^2+\theta(1-\theta)(|\bm E_1|^2+|\bm E_2|^2)\\
    &=\theta|\bm E_1|^2+(1-\theta)|\bm E_2|^2.
\end{align*}
Equality holds if and only if $\bm E_1=\bm E_2$, which is evident.
\end{proof}

Consider the minimization problem of the above convex electrostatic energy with the Gauss's law constraint~\eqref{gaussn}. Since the electrostatic energy functional~\eqref{energyE} is convex and the Gauss’s law constraint is linear, the associated constrained minimization problem is a convex optimization problem.
Consequently, the Karush–Kuhn–Tucker (KKT) conditions \cite{izmailov2003karush} are necessary and sufficient condition for minimizer. The Lagrangian functional of the optimization problem reads
\begin{equation*}
    \mathcal{L}[\bm E, \phi] := \int_{\Omega} \frac{\varepsilon|\bm E|^2}{2} d \bm x - \int_{\Omega} \phi (\nabla \cdot (\varepsilon\bm E) - \rho) \, d \bm x,
\end{equation*}
where $\phi$ is the Lagrange multiplier. By the the stationarity condition and the primal feasibility condition, one has
\begin{equation*}
    \bm E = - \nabla \phi \quad \text{and} \quad \nabla \cdot (\varepsilon\bm E) = \rho.
\end{equation*}
This indicates that the Lagrange multiplier $\phi$  corresponds precisely to the electrostatic potential, and  the electric field at the minimizer satisfies the curl-free condition $\nabla\times \bm E=0$.


In fact, it can be shown that the electric field $\bm E$ resulting from the minimization of the above energy functional coincides with the negative gradient of the weak solution of the Poisson’s equation in the appropriate Sobolev space.
For simplicity, we assume that  $\Omega=(0,L_x)\times(0,L_y)$ is a domain with
periodic boundary conditions. Let $L_{per}^{p}(\Omega)$ and $H_{per}^{k}(\Omega)$ denote the space of all $\overline{\Omega}$-periodic functions on $\mathbb{R}^n$ whose restrictions onto $\Omega$ are in $L^{p}(\Omega)$ and $H^{k}(\Omega)$ respectively. Define
\begin{align*}
    &\mathring{L}_{per}^{p}(\Omega)=\{f(\bm x)\in L_{per}^{p}(\Omega):\mathscr{A}_{\Omega}(f)=0\},\\
    &\mathring{H}_{per}^{k}(\Omega)=\{f(\bm x)\in H_{per}^{k}(\Omega):\mathscr{A}_{\Omega}(f)=0\},
\end{align*}
where $\mathscr{A}_{\Omega}(f)$ denote the average 
\begin{equation*}
    \mathscr{A}_{\Omega}(f):=\frac{1}{|\Omega|}\int_{\Omega}f(\bm x)\ d \bm x.
\end{equation*}
The space of all $\overline{\Omega}$-periodic $C^k$-functions  are denote by $\ C_{per}^{k}(\overline{\Omega})$  for any $k\in\mathbb{Z}$, further define
\begin{align*}
    &H(\operatorname{div} ,\Omega)=\{\bm E\in L^{2}(\Omega):\nabla\cdot \bm E\in L^{2}(\Omega)\},\\
    &H_{per}(\operatorname{div},\Omega)=\text{the}\ H(\operatorname{div},\Omega)\text{-closure of}\ C_{per}^{1}(\overline{\Omega}) \ \text{functions restricted to}\ \Omega. 
\end{align*}
Given $\rho\in L_{per}^{2}(\Omega)$ and $\varepsilon \in L^{\infty}_{per}(\Omega)$, denote 
\begin{equation}
    S_{\rho}=\{\bm E\in H_{per}(\operatorname{div},\Omega):\nabla\cdot \varepsilon \bm E=\rho\}.
\end{equation}
Assume that $\rho\in\mathring{L}^{2}_{per}(\Omega)$ and $\varepsilon$ satisfies 
\begin{equation}\label{epsilon}
    0<\varepsilon_{\text{min}}<\varepsilon(\bm x)<\varepsilon_{\text{max}},\quad \forall \bm x \in \mathbb{R}^n,
\end{equation}
where $\varepsilon_{\text{min}}\  \text{and}\ \varepsilon_{\text{max}}$ are the lower and upper bounds. Theorem~\ref{th1} was provided in \cite{li2024finite}.

\begin{thm} (\cite{li2024finite})\label{th1}
    Let $\rho\in\mathring{L}^{2}_{per}(\Omega)$ and $\varepsilon \in L^{\infty}_{per}(\Omega)$ satisfy~\eqref{epsilon}. Define 
    \begin{align} &I[\phi]=\int_{\Omega}\bigg(\frac{\varepsilon}{2}|\nabla\phi|^{2}-\rho\phi\bigg)\ d \bm x\quad\forall\phi\in H_{per}^{1}(\Omega),\\
&F[\bm E]=\int_{\Omega}\frac{\varepsilon}{2}|\bm E|^{2}\ d \bm x\quad \forall \bm E\in S_{\rho}.\label{aa}
\end{align}
Then one has,
    \begin{enumerate}
        \item There exists a unique $\phi_{min}\in \mathring{H}^{1}_{per}(\Omega)$ such that $I[\phi_{min}]=\operatorname{min}_{\phi\in \mathring{H}^{1}_{per}(\Omega)}I[\phi]$, and it is the unique weak solution in $\mathring{H}^{1}_{per}(\Omega)$ to Poisson's equation~\eqref{poi}.
        \item There exists a unique $\bm E_{min}\in S_{\rho}$ such that $F[\bm E_{min}]=\operatorname{min}_{\bm E\in S_{\rho}}F[\bm E]$. Moreover, one has $\bm E_{min}=-\nabla\phi_{min}.$
    \end{enumerate}
\end{thm}

This theorem demonstrates that solving the Poisson's equation~\eqref{poi} is equivalent to solving the minimization problem in~\eqref{aa}. And the condition enforcing restriction $\bm E\in S_{\rho}$ is precisely the Gauss's law constraint~\eqref{gaussn}.

\section{Local relaxation method}\label{la}
For simplicity of presentation, all theoretical derivations, algorithmic steps, and notation definitions in this section are explicitly limited to two-dimensional configurations. It is crucial to emphasize that the proposed numerical algorithms are applicable to three-dimensional problems without introducing any new challenges and extra structural adjustments, and corresponding 3D numerical experiments are included in Section~\ref{NumResults}.
  
\subsection{Discretization}
The computational domain $\Omega=(0,L_x)\times(0,L_y)$ is covered by a uniform grid with grid spacings $\Delta x$ and $\Delta y$ in two coordinate directions. Let $\Delta \Omega = \Delta x \Delta y$ denote the area of a cell. Let $N_x$ and $N_y$  be the number of grid points in each dimension.  
Denote  $\Delta x\mathbb{Z}\times\Delta y\mathbb{Z}=\{(i\Delta x,j\Delta y):i,j\in\mathbb{Z}\}$ and $\Delta x(\mathbb{Z}+1/2)\times\Delta y(\mathbb{Z}+1/2)=\{((i+\frac{1}{2})\Delta x,(j+\frac{1}{2})\Delta y):i,j\in\mathbb{Z}\}$.
For any grid function $f:\Delta x\mathbb{Z}\times\Delta y\mathbb{Z}\rightarrow\mathbb{R}$, one denotes $f_{ij}=f(i\Delta x,j\Delta y)$. 
For $\overline{\Omega}$-periodic grid functions, one defines spaces
\begin{align}
    &V_{h}=\{\overline{\Omega}\text{-periodic grid functions }f :\Delta x\mathbb{Z}\times\Delta y\mathbb{Z} \rightarrow \mathbb{R}\},\\
   \text{and}\quad  &\mathring{V}_{h}=\{f\in V_h:\mathscr{A}_{h}(f)=0\},
\end{align}
where $\mathscr{A}_{h}$ denotes the discrete average 
\begin{equation}
    \mathscr{A}_{h}(f)=\frac{1}{N_xN_y}\sum_{i=0}^{N_x-1}\sum_{j=0}^{N_y-1}f_{ij}.
\end{equation}

Given $\bm E=(E_x,E_y)$, one defines a discrete electric field as a vector-valued function 
\begin{equation}
\mathbb{E}_{i+\frac{1}{2},j+\frac{1}{2}}=(E_{i+\frac{1}{2},j},E_{i,j+\frac{1}{2}})\quad \forall i,j\in\mathbb{Z},
\end{equation}
where $E_{i+\frac{1}{2}, j}$ represents the value of $E_x$ at the grid point $((i + \frac{1}{2}) \Delta x, j \Delta y)$, and  $E_{i, j+\frac{1}{2}}$ represents the value of $E_y$ at the grid point $(i \Delta x, (j+ \frac{1}{2}) \Delta y)$. Denote
\begin{equation}
     Y_{h}=\{\overline{\Omega}-\text{periodic functions }  \mathbb{E}:\Delta x(\mathbb{Z}+1/2)\times\Delta y(\mathbb{Z}+1/2) \rightarrow \mathbb{R}^2\}.
\end{equation}
For any $\mathbb{E}\in Y_h$, define the discrete average 
\begin{equation*} \mathscr{A}_h(\mathbb{E})=\frac{1}{N_xN_y}\sum_{i=0}^{N_x-1}\sum_{j=0}^{N_y-1}(E_{i+\frac{1}{2},j},E_{i,j+\frac{1}{2}}),
\end{equation*}
and the norm
\begin{equation*}
||\mathbb{E}||_{h}=\sqrt{\Delta\Omega\sum_{i=0}^{N_x-1}\sum_{j=0}^{N_y-1}\left(E_{i,j+\frac{1}{2}}^2+E_{i+\frac{1}{2},j}^2\right)}.
\end{equation*}
The discrete divergence and curl on $Y_h$ are defined by

\begin{align*}(\nabla_h\cdot\mathbb{E})_{i,j}&=\frac{E_{i+\frac{1}{2},j}-E_{i-\frac{1}{2},j}}{\Delta x}+\frac{E_{i,j+\frac{1}{2}}-E_{i,j-\frac{1}{2}}}{\Delta y},\\
(\nabla_h\times\mathbb{E})_{i+\frac{1}{2},j+\frac{1}{2}}&=\frac{E_{i+1,j+\frac{1}{2}}-E_{i,j+\frac{1}{2}}}{\Delta x}-\frac{E_{i+\frac{1}{2},j+1}-E_{i+\frac{1}{2},j}}{\Delta y}.
\end{align*}

Given $\varepsilon\in C_{per}(\overline{\Omega})$ satisfying~\eqref{epsilon}, define a new average on $\Delta x(\mathbb{Z}+1/2)\times\Delta y(\mathbb{Z}+1/2)$ by
\begin{equation*}
    \varepsilon_{i+\frac{1}{2},j}=\frac{\varepsilon_{i,j}+\varepsilon_{i+1,j}}{2},\quad\varepsilon_{i,j+\frac{1}{2}}=\frac{\varepsilon_{i,j}+\varepsilon_{i,j+1}}{2}.
\end{equation*}
Given $\rho^{h} \in \mathring{V}_{h}$ the discrete charge density, Gauss's law \eqref{gaussn} can be discretized as
\begin{equation}\label{discretgauss}
\nabla_{h}\cdot \left(\varepsilon\mathbb{E}\right)=\rho^{h}.
\end{equation}
And the electrostatic energy~\eqref{energyE} admits an approximation of the form 
\begin{equation}\label{energyEdis}
    \mathcal{F}^{h}[\mathbb{E}] = \frac{\Delta\Omega}{2}\sum_{i=0}^{N_x-1}\sum_{j=0}^{N_y-1}(\varepsilon_{i,j+\frac{1}{2}} E_{i,j+\frac{1}{2}}^2+\varepsilon_{i+\frac{1}{2},j}E_{i+\frac{1}{2},j}^2),
\end{equation}
which is to be minimized under the constraint of the discrete Gauss's law \eqref{discretgauss}.

\subsection{Single mesh relaxation}
The local algorithm begins with initializing an electric field $\mathbb{E}^{(0)}$ that satisfies the discrete form of Gauss's law~\eqref{discretgauss}. In certain applications where the Poisson's equation must be solved at every time step, the electric field that satisfies Gauss’s law can be obtained through time evolution from the electric field at the previous time step. In contrast, in a broader range of scenarios, this electric field is unknown a priori and thus requires explicit computation. Notably, the procedure for constructing an electric field $\mathbb{E}^{(0)}$ is not unique, one feasible approach is presented in \cite{BSD:PRE:2009}. Let $\overline{\rho}_{j}=\sum_{i=0}^{N_{x}-1}\rho^h_{ij}/N_{x}$ denote the average charge over the $j$th $x$-aligned line. For $i = 0, \cdots, N_x-1$, the $y$-component of the electric field is then computed by
\begin{equation}\label{getE1}
    \varepsilon_{i,j+\frac{1}{2}} E_{i,j+\frac{1}{2}} = \varepsilon_{i,j-\frac{1}{2}} E_{i,j-\frac{1}{2}}+\Delta y \overline{\rho}_{j}, \quad  \ j = 1, \cdots, N_y-1,
\end{equation}
with $E_{i,\frac{1}{2}}=0$ for all $i=0,\cdots,N_x-1.$ Subsequently, the $x$-directional component of the electric field is updated by manipulating the discrete Gauss’s law~\eqref{discretgauss}. For $j = 0, \cdots, N_y-1$, the computation follows
\begin{equation}\label{getE3}
   \varepsilon_{i+\frac{1}{2},j} E_{i+\frac{1}{2},j} = \varepsilon_{i-\frac{1}{2},j} E_{i-\frac{1}{2},j} + \Delta x (\rho^h_{i,j}-\overline{\rho}_{j}), \quad i = 1, \cdots, N_x-1,  
\end{equation}
with $E_{\frac{1}{2},j}=0, ~j=0,\cdots,N_y-1.$ 
Crucially, the discrete electric field $\mathbb{E}^{(0)}$
constructed through the sequential application of \eqref{getE1} and \eqref{getE3} is guaranteed to satisfy the discrete Gauss’s law~\eqref{discretgauss}, aligning with the initial requirement for the algorithm.  

The next step of the local algorithm involves iteratively updating the electric field to minimize the discrete energy functional $\mathcal{F}^{h}$~\eqref{energyEdis} with constraint of Gauss's law. This minimization ensures the electric field to be curl-free, thus coinciding with the exact solution to Poisson’s equation. It is emphasized that the local algorithm naturally and rigorously preserves Gauss’s law, which is achieved by designing a specialized numerical update scheme for electric field. 
The minimization of $\mathcal{F}^{h}$ is implemented via local updates \cite{M:JCP:2002,maggs2002local} to the discrete electric field within each computational cell, so we call it as single mesh relaxation method. Starting from the initial field ${\mathbb{E}^{(0)}}$, the electric field is updated by introducing a rotational flux $\eta$ along the four edges of each cell, illustrated in Fig.~\ref{updatepla}. Notably, this same flux $\eta$ is chosen uniformly across the four edges of a single cell — a specific design constraint that preserves the discrete Gauss’s law strictly. For details, the process is as follows:
\begin{equation}\label{change0}
    \begin{aligned}
E_{i,j+\frac{1}{2}} &\leftarrow E_{i,j+\frac{1}{2}}-\frac{\eta}{\varepsilon_{i,j+\frac{1}{2}}\Delta x}, \\
E_{i+1,j+\frac{1}{2}} &\leftarrow E_{i+1,j+\frac{1}{2}}+\frac{\eta}{\varepsilon_{i+1,j+\frac{1}{2}}\Delta x},\\
E_{i+\frac{1}{2},j} &\leftarrow E_{i+\frac{1}{2},j}+\frac{\eta}{\varepsilon_{i+\frac{1}{2},j}\Delta y},\\
E_{i+\frac{1}{2},j+1} &\leftarrow E_{i+\frac{1}{2},j+1}-\frac{\eta}{\varepsilon_{i+\frac{1}{2},j+1}\Delta y}.
\end{aligned}
\end{equation}

\begin{figure}[H]
  \centering
  \includegraphics[width=0.7\linewidth]{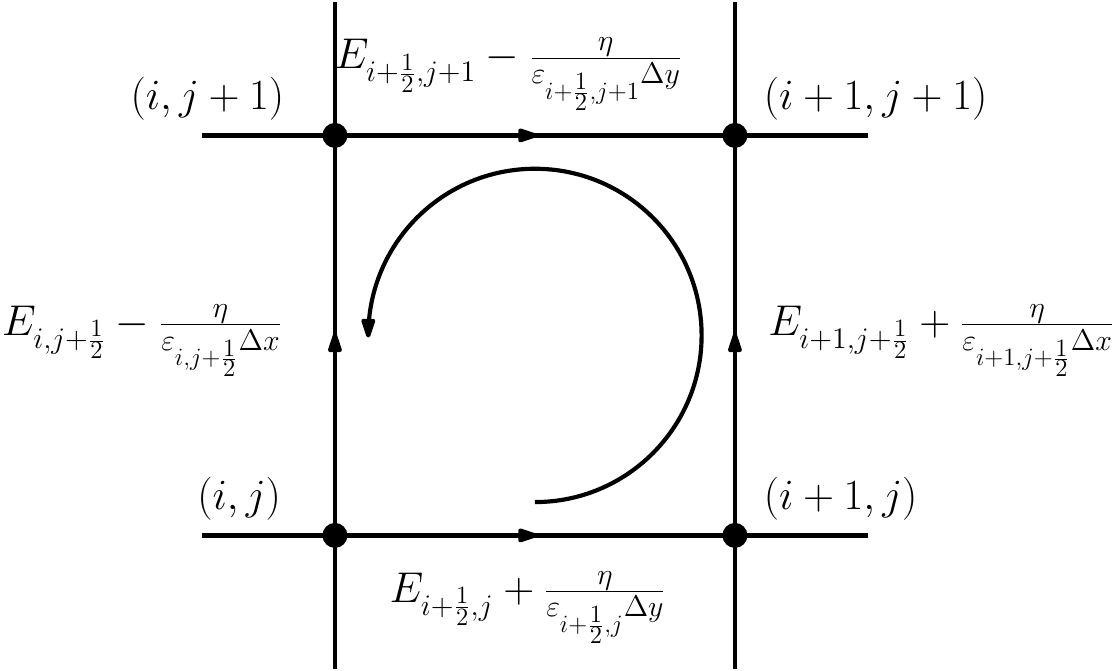}
  \caption{Schematic of the electric-field update on a single grid cell $(i,j)$ with flux $\eta$. }
  \label{updatepla}
\end{figure}
It can be simply proved that discrete Gauss's law~\eqref{discretgauss} defined at the four nodes of this updating cell still strictly holds. With this update, the energy change is given by
\begin{equation}\label{delta}
\begin{aligned}
\delta \mathcal{F}^{h}[\eta]=& \frac{\eta^2}{2} \left[ \frac{\Delta x}{ \Delta y} \left( \frac{1}{\varepsilon_{i+\frac{1}{2},j}} + \frac{1}{\varepsilon_{i+\frac{1}{2},j+1}}\right) +\frac{\Delta y}{\Delta x} \left( \frac{1}{\varepsilon_{i,j+\frac{1}{2}}} + \frac{1}{\varepsilon_{i+1,j+\frac{1}{2}}} \right) \right]  \\
    &+ \eta \bigg[\Delta x\bigg( E_{i+\frac{1}{2},j} - E_{i+\frac{1}{2},j+1}\bigg)+ \Delta y \left( E_{i+1,j+\frac{1}{2}} - E_{i,j+\frac{1}{2}}\right) \bigg].
\end{aligned}  
\end{equation}
This is a quadratic form of $\eta$, thus the minimum value can be exactly obtained when
\begin{equation}\label{eta}
    \eta = -\frac{\Delta y (\Delta x)^2 \left(E_{i+\frac{1}{2},j} - E_{i+\frac{1}{2},j+1}\right) + \Delta x (\Delta y)^2 \left( E_{i+1,j+\frac{1}{2}}- E_{i,j+\frac{1}{2}}\right)}
    {(\Delta x)^2 \left( \frac{1}{\varepsilon_{i+\frac{1}{2},j}} + \frac{1}{\varepsilon_{i+\frac{1}{2},j+1}}\right) + (\Delta y)^2 \left( \frac{1}{\varepsilon_{i,j+\frac{1}{2}}} + \frac{1}{\varepsilon_{i+1,j+\frac{1}{2}}} \right)}.
\end{equation}
This formula provides the optimal flux for the change of electric field. Such a local update traverses all discrete grid cells, and iterative updates continue until the stopping condition is met.

The single mesh update is not sufficient \cite{levrel2005monte} to ensure ergodicity in the sampling for Monte Carlo simulation. It may not force the electrostatic energy \eqref{energyEdis} reach to its minimum~\cite{li2024finite}. To address this issue, a line shift update \cite{jiang2018improved,li2024finite} has been  introduced, which uses the principle of ensuring that the flux entering a discrete node is equal to that exiting it. This design enables the strict preservation of the discrete Gauss’s law, while the optimal flux value is derived from energy minimization calculations.
Specifically, for $x$-direction, the update of electric field at $j$th $x$ line with flux $\eta_x$ is
\begin{equation}\label{updateetax}
E_{i+\frac{1}{2},j} \leftarrow E_{i+\frac{1}{2},j}+\frac{\eta_x}{\varepsilon_{i+\frac{1}{2},j}}, \qquad i=0, 1,...,N_x-1.
\end{equation}
The energy change after this update is
\begin{equation}
\delta \mathcal{F}^h[\eta_x]=\Delta x\Delta y
\bigg[\frac{\eta_x^2}{2}\bigg(\sum_{i=0}^{N_x-1} \frac{1}{\varepsilon_{i+\frac{1}{2}, j}}\bigg)+\eta_x\bigg(\sum_{i=0}^{N_x-1} E_{i+\frac{1}{2}, j}\bigg)\bigg].
\end{equation}
The energy reaches its minimum value, when
\begin{equation}\label{eta_x}
    \eta_x = -\sum_{i=0}^{N_x-1} E_{i+\frac{1}{2}, j}\bigg{/}\sum_{i=0}^{N_x-1} \frac{1}{\varepsilon_{i+\frac{1}{2}, j}}.
\end{equation}
Similarly, for $y$-direction, the process of update at $i$th $y$ line with $\eta_y$ is 
\begin{equation}\label{updateetay}
E_{i,j+\frac{1}{2}} \leftarrow E_{i,j+\frac{1}{2}}+\frac{\eta_y}{\varepsilon_{i,j+\frac{1}{2}}}, \qquad j=0, 1,...,N_y-1.
\end{equation}
The optimal flux $\eta_y$ is
\begin{equation}\label{eta_y}
    \eta_y = -\sum_{j=0}^{N_y-1} E_{i, j+\frac{1}{2}}\bigg{/}\sum_{j=0}^{N_y-1} \frac{1}{\varepsilon_{i,j+\frac{1}{2}}}.
\end{equation}

To summarize, the first step of the local algorithm involves the construction of an initial electric field that strictly adheres to the discrete Gauss’s law. The subsequent step includes the minimization of the electrostatic energy functional, which is tailored to achieve a curl-free electric field configuration while rigorously preserving the discrete Gauss’s law throughout the entire process. 
The local algorithm with single mesh relaxation is outlined in Algorithm~\ref{alg:plaquette}. 

\begin{algorithm}[H]
\caption{Local Algorithm with Single Mesh  Relaxation}
\label{alg:plaquette}

\begin{algorithmic}[1]
\State \textbf{Step 1.} Initialize field $\mathbb{E}^{(0)}$to satisfy the discrete Gauss's law~\eqref{discretgauss}. Set $m = 0$.

\Statex
\State \textbf{Step 2.} Curl-free relaxation $\mathbb{E} := \mathbb{E}^{(m)}$.
\State Update $\mathbb{E}$ along each cell via \eqref{change0}.
    \State Update $\mathbb{E}$ by the line shift via ~\eqref{updateetax} and ~\eqref{updateetay}.
\Statex
\State \textbf{Step 3.} Termination check:
    \If{$\delta\mathcal{F}^h<\varepsilon_{tol}$}
        \State \textbf{stop}
    \Else
        \State Set $\mathbb{E}^{(m+1)} = \mathbb{E}$, $m \leftarrow m + 1$.
        \State \textbf{Go to} Step 2
    \EndIf
\end{algorithmic}
\label{a0}
\end{algorithm}

\begin{remark}
For non-periodic boundaries, we employ the image charges method (ICM) \cite{liang2021high}  to handle Neumann and Dirichlet conditions. By mapping the system to an equivalent periodic configuration, the electric field can then be computed using the local method developed here. 
\end{remark}

\subsection{Hierarchical mesh relaxation}\label{algos}

The single mesh relaxation updates the electric field sequentially in each cell of the finest computational grid. A critical limitation of this approach, however, is that it yields only slow reduction of rotational errors with particularly in low-frequency error components, thereby restricting the inefficient rate at which the electrostatic energy can be minimized. Recall that the core principle of such local algorithms is to design specialized update fluxes that drive the most rapid possible reduction in electrostatic energy, while rigorously preserving the discrete Gauss’s law at all times. In fact, without considering the curl-free condition~\eqref{curlfreen}, the general solution of Eq.~\eqref{gaussn} can be represented as 
\begin{equation*}
\varepsilon\bm E=-\varepsilon\nabla\phi+\nabla\times \mathbf{Q},
\end{equation*}
where $\phi$ satisfies $-\nabla\cdot (\varepsilon \nabla \phi)=\rho$, and $\mathbf{Q}$ is the rotational degree of freedom. Based on the orthogonality of the two components of $\bm E$, the electrostatic energy~\eqref{energyE} can be written as
\begin{equation*}
   \mathcal{F}_{\text{pot}}[\bm E] := \int_{\Omega} \frac{\varepsilon}{2}[(\nabla\phi)^2+(\nabla\times\mathbf{Q})^2] d \bm x.
\end{equation*}
From this, local algorithms essentially keep the gradient terms, namely, Gauss's law unchanged, and the minimization of energy is equivalent to the minimization of the energy associated with rotational term.
Inspired by geometric multigrid methods~\cite{hemker1981introduction,mccormick1987multigrid}, a well-established technique that accelerates the iterative solution of linear systems by leveraging a hierarchy of computational grids with different spatial scales, we propose a HLR method  by integrating a multi-level hierarchical structure into the framework. This novel design is expected to effectively accelerate the propagation of rotational degrees of freedom across the entire computational domain while keeping the principle of local algorithms.

In the HLR, the computational domain $\Omega$ is partitioned into a set of nested grid blocks across multiple refinement levels $L_k,~k=1,\ldots,M$, where $L_1$ represents the coarsest level and $L_M$ the finest level. In specific, the grid partitioning at level $L_k$ is defined as:
$$
\Omega=\bigcup_{m,n=1}^{2^{k}} \Omega_{m,n}^{(k)},
$$
where $\Omega_{m,n}^{(k)}$ denotes the $(m,n)$-th grid block at refinement level $L_k$, with each block at level $L_k$ being partitioned into $2\times2$ sub-blocks to form the grid at level $L_{k+1}$ (i.e. $\Omega_{m,n}^{(k)}=\cup _{a,b=1}^2\Omega_{2m-a,2n-b}^{(k+1)}$). Notice that $2^M=N_x=N_y$ if the same resolutions are used in $x$- and $y$- coordinates.

Based on this hierarchical grid structure, we propose two HLR methods, forward HLR and zigzag HLR, within the local algorithm framework, as illustrated in  Fig.~\ref{dd} (a) for a single iteration. Unlike the single mesh relaxation method, where curl-free relaxation is restricted to the finest grid ($L_M$), the HLR method shifts the curl-free relaxation from individual finest grid cells to grid blocks at different levels. Performing an update on any grid block $\Omega^{(k)}_{m,n}$ implies executing a local process on the grid nodes covered by that block at its corresponding resolution.

The forward HLR method follows a strictly sequential, layer-by-layer update path from coarse scales to fine scales grids. Starting at the coarsest grid level $L_1$ (e.g., the initial $2\times 2$ block grid), the algorithm completes the field update on the current coarse level, then proceeds forward and irreversibly to the next finer level $L_2$. After finishing updates on $L_2$, it moves directly to the finer level $L_3$, and repeats this process until reaching the finest grid level $L_{\max}$.
This strategy follows a one-way progressive pathway:
$$L_1\to L_2\to \cdots \to L_M.$$
No backward revisit or correction to coarser levels is performed during the entire procedure.

In contrast, the zigzag HLR scheme adopts a reciprocal, backtracking update pathway that repeatedly travels from coarse to fine, then back to coarse, and forward to fine again, forming a 
zigzag trajectory across grid levels.
After updating the coarsest level $L_1$ and refining to the finer levels $L_2$ and $L_3$, the algorithm intentionally revisits the coarser level $L_2$ to apply secondary corrections based on the fine-scale information. It then refines again to $L_3$ and proceeds toward $L_4$, with similar back-and-forth corrections at each intermediate level. Its typical pathway can be summarized as:
$$L_1\to L_2 \to L_3\to L_2\to L_3\to L_4\to \cdots \to L_M.$$
This zigzag backtracking enables two-way information exchange between coarse and fine grids, rather than the one-way propagation used in the forward scheme.




\begin{figure}[ht]
    \centering 
    \captionsetup[sub]{labelformat=empty}
    \begin{subfigure}{0.6\textwidth} 
        \centering
        \includegraphics[width=\linewidth]{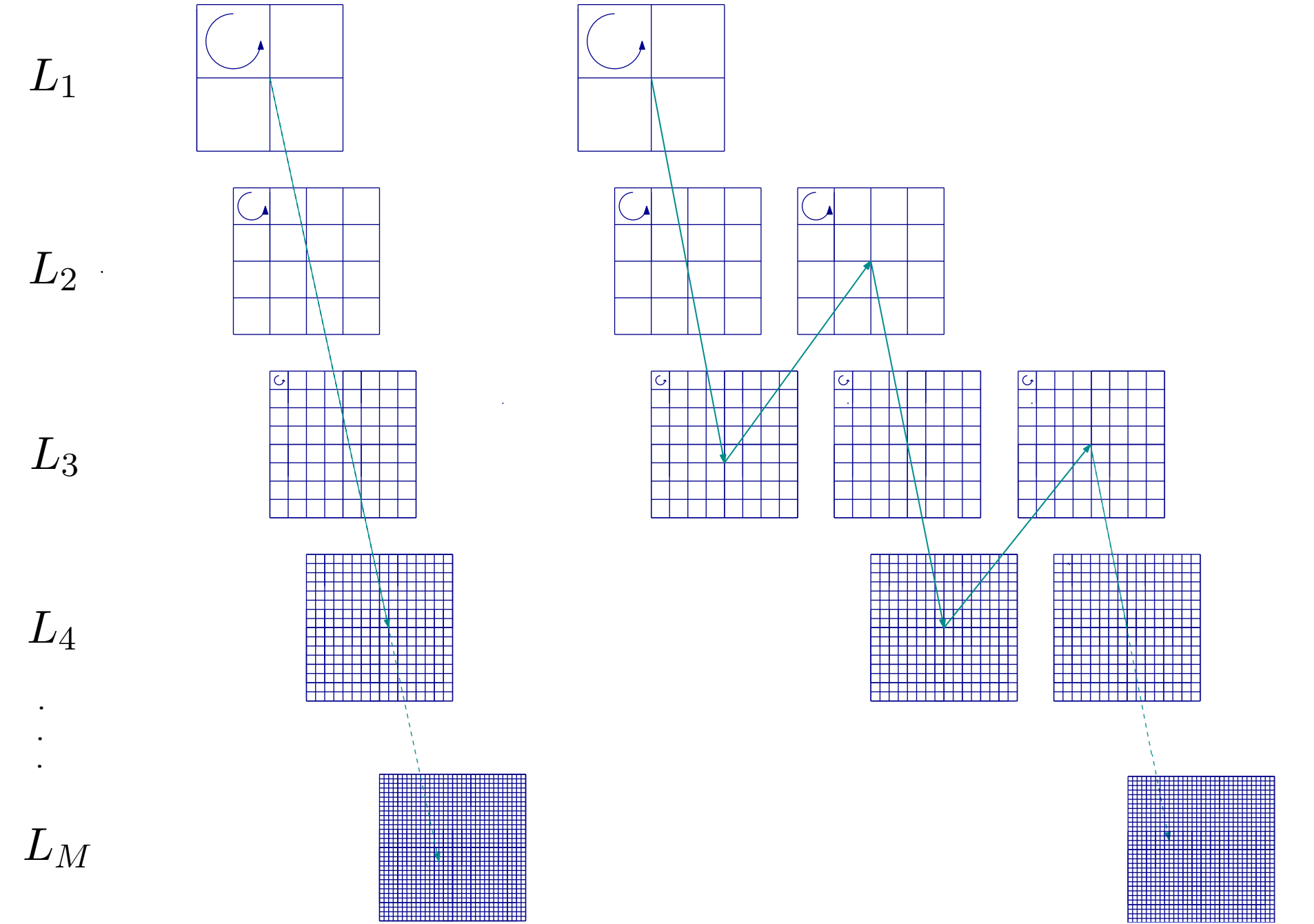} 
        \caption{(a)} 
        \label{dd1}
    \end{subfigure}
    \begin{subfigure}{0.36\textwidth}
        \centering \includegraphics[width=\linewidth]{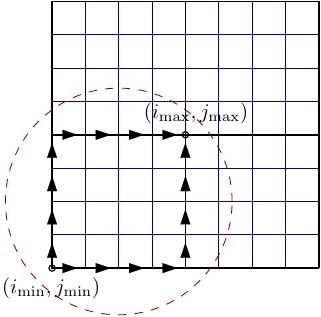}
        \caption{(b)}
        \label{dd2}
       
    \end{subfigure}
    \caption{(a) Schematic illustration of relaxation on hierarchical meshes with forward (left), and zigzag (right) HLR. (b) Schematic of the electric-field update on a coarse-grid block covering multiple finest cells.}
\label{dd}
    \end{figure}

The procedures for local curl-free relaxation on coarse grids are analogous to those on finest grids described in the previous section. At the grid level $L_k$, the curl-free relaxation is performed on grid blocks $\Omega^{(k)}_{m,n},~m,n=1,\cdots,2^k$,
where each block is uniquely specified by its lower-left coordinate $(i_{\min},j_{\min})=\left(2^{M-k}(m-1),2^{M-k}(n-1)\right)$ and upper-right corner coordinate 
$(i_{\max},j_{\max})=(2^{M-k}m,2^{M-k}n)$ in the global finest-grid index system.
For any given block $\Omega^{(k)}_{m,n},~m,n=1,\cdots,2^k$, the curl-free relaxation update is applied exclusively to all electric field components defined on the block boundary edges, using the same rotational flux–based strategy established earlier. As illustrated in Fig.~\ref{dd} (b), the update procedure for a block with lower-left coordinate $(i_{\min}, j_{\min})$ and upper-right coordinate $(i_{\max}, j_{\max})$ proceeds as follows:

\begin{equation}\label{changecu}
    \begin{aligned}
E_{i_{\min},j+\frac{1}{2}} &\leftarrow E_{i_{\min},j+\frac{1}{2}}-\frac{\eta}{\varepsilon_{i_{\min},j+\frac{1}{2}}\Delta x},\quad j_{\min}\leq j<j_{\max}, \\
E_{i_{\max},j+\frac{1}{2}} &\leftarrow E_{i_{\max},j+\frac{1}{2}}+\frac{\eta}{\varepsilon_{i_{\max},j+\frac{1}{2}}\Delta x},\quad j_{\min}\leq j<j_{\max},\\
E_{i+\frac{1}{2},j_{\min}} &\leftarrow E_{i+\frac{1}{2},j_{\min}}+\frac{\eta}{\varepsilon_{i+\frac{1}{2},j_{\min}}\Delta y},\quad i_{\min}\leq i<i_{\max},\\
E_{i+\frac{1}{2},j_{\max}} &\leftarrow E_{i+\frac{1}{2},j_{\max}}-\frac{\eta}{\varepsilon_{i+\frac{1}{2},j_{\max}}\Delta y},\quad i_{\min}\leq i<i_{\max}.
\end{aligned}
\end{equation}

This update preserves the discrete Gauss’s law defined on all grids of the boundary of this block. Similarly, the energy change after the electric displacement update is
\begin{equation}\label{deltacu}
\begin{aligned}
\delta \mathcal{F}^h[\eta]=& \frac{\eta^2}{2} \left[ \frac{\Delta x}{ \Delta y} \sum_{i=i_{\min}}^{i_{\max}-1}\left( \frac{1}{\varepsilon_{i+\frac{1}{2},j_{\min}}} + \frac{1}{\varepsilon_{i+\frac{1}{2},j_{\max}}}\right) +\frac{\Delta y}{\Delta x} \sum_{j=j_{\min}}^{j_{\max}-1}\left( \frac{1}{\varepsilon_{i_{\min},j+\frac{1}{2}}} + \frac{1}{\varepsilon_{i_{\max},j+\frac{1}{2}}} \right) \right]  \\
    &+ \eta \bigg[\Delta x\sum_{i=i_{\min}}^{i_{\max}-1}\bigg( E_{i+\frac{1}{2},j_{\min}} - E_{i+\frac{1}{2},j_{\max}}\bigg)+ \Delta y \sum_{j=j_{\min}}^{j_{\max}-1}\left( E_{i_{\max},j+\frac{1}{2}} - E_{i_{\min},j+\frac{1}{2}}\right) \bigg].
\end{aligned}  
\end{equation}
The maximization of the energy change gives the exact solution of $\eta$:
\begin{equation}\label{etacu}
    \eta = -\frac{\Delta y (\Delta x)^2 \sum_{i=i_{\min}}^{i_{\max}-1}\bigg( E_{i+\frac{1}{2},j_{\min}} - E_{i+\frac{1}{2},j_{\max}}\bigg) + \Delta x (\Delta y)^2 \sum_{j=j_{\min}}^{j_{\max}-1}\left( E_{i_{\max},j+\frac{1}{2}} - E_{i_{\min},j+\frac{1}{2}}\right)}
    {(\Delta x)^2 \sum_{i=i_{\min}}^{i_{\max}-1}\left( \frac{1}{\varepsilon_{i+\frac{1}{2},j_{\min}}} + \frac{1}{\varepsilon_{i+\frac{1}{2},j_{\max}}}\right) + (\Delta y)^2 \sum_{j=j_{\min}}^{j_{\max}-1}\left( \frac{1}{\varepsilon_{i_{\min},j+\frac{1}{2}}} + \frac{1}{\varepsilon_{i_{\max},j+\frac{1}{2}}} \right)}.
\end{equation}

The local algorithms with HLR methods are summarized in Algorithm~\ref{alg:multigrid}.


\begin{algorithm}[H]
\caption{Local Algorithm with HLR Methods.}
\label{alg:multigrid}

\begin{algorithmic}[1]
\State \textbf{Step 1.} Initialize displacement $\mathbb{E}^{(0)}$ satisfy the discrete Gauss's law~\eqref{discretgauss}. Set $m = 0$.

\Statex
\State\textbf{Step 2.} Curl-free Relaxation $\mathbb{E} := \mathbb{E}^{(m)}$.

\Statex\underline{\textbf{Forward}}:
\For{$\text{Level}:~L_k=L_1$ \textbf{to} $L_M$}
        \For{Each block: $m,n=1$ \textbf{to} $2^k$}
            \State Update $\mathbb{E}$ to get $\mathbb{E}^{*}$ via \eqref{changecu} and  $\mathbb{E} \leftarrow \mathbb{E}^{*}$.
            
        \EndFor
    \EndFor
    
\Statex\underline{\textbf{Zigzag}}:
\For {$\text{Level}:~L_\ell=L_1$ \textbf{to} $L_{M-2}$}
    \For{$\text{Level}:~L_k=L_\ell$ \textbf{to} $L_{\ell+2}$}
        \For{Each block: $m,n=1$ \textbf{to} $2^k$}
            \State Update $\mathbb{E}$ to get $\mathbb{E}^{*}$ via \eqref{changecu} and  $\mathbb{E} \leftarrow \mathbb{E}^{*}.$
            \EndFor
        \EndFor
    \EndFor
  \State Update $\mathbb{E}$ by the line shift via ~\eqref{updateetax} and ~\eqref{updateetay}.
\Statex
\State \textbf{Step 3.} Termination check:
    \If{$\delta\mathcal{F}^h<\varepsilon_{tol}$}
        \State \textbf{stop}
    \Else
        \State Set $\mathbb{E}^{(m+1)} = \mathbb{E}$, $m \leftarrow m + 1$.
        \State \textbf{Go to} Step 2
    \EndIf
\end{algorithmic}
\label{a1}
\end{algorithm}

\section{Convergence Analysis}
~\label{cn}
 
In this section, we will present the proof of convergence for the new proposed HLR relaxation  method, following the idea of~\cite{li2024finite}. For convenience, one assumes that $L_x=L_y=L$ and $N_x=N_y=N=2^M$, such that $\Delta x=\Delta y= h=L/{2^M}$. Thus, one can denote $h\mathbb{Z}^2=\Delta x\mathbb{Z}\times\Delta y\mathbb{Z}$ and $h(\mathbb{Z}+1/2)^2=\Delta x(\mathbb{Z}+1/2)\times\Delta y(\mathbb{Z}+1/2)$. Given $\varepsilon\in Y_h$ and $\rho^h\in \mathring{V}_h$, define
\begin{equation}
    S_{\rho,h}=\{\mathbb{E}\in Y_{h}:\nabla_h\cdot \left(\varepsilon \mathbb{E}\right)=\rho^{h}\ \text{on}\ h\mathbb{Z}^2\}.
\end{equation}
Thus, the discrete electrostatic energy \eqref{energyEdis} can be treated as a functional $\mathcal{F}^{h}:S_{\rho,h}\rightarrow\mathbb{R}$. Lemma~\ref{lem2} and~\ref{lem4} were obtained in \cite{li2024finite}, which provide some equivalent conditions on the minimizer of  $\mathcal{F}^{h}$.

\begin{lem}\label{lem2}
    Let $\varepsilon\in Y_{h}$ satisfy~\eqref{epsilon} and $\rho^{h}\in \mathring{V}_{h}$. There exists a unique minimizer $\mathbb{E}_{min}^{h}$ of $\mathcal{F}^{h}:S_{\rho,h}\rightarrow\mathbb{R}$. For $\mathbb{E}\in S_{\rho,h}$, $\mathbb{E}=\mathbb{E}_{min}^{h}$ if and only if:
    \begin{enumerate}
        \item Local equilibrium: $\mathbb{E}$ is curl free, i.e, $\nabla_{h}\times \mathbb{E}=0$ on $h(\mathbb{Z}+1/2)^{2}$;
        \item Zero total field : $\mathscr{A}_{h}(\mathbb{E})=0$ in $\mathbb{R}^{2}$.
    \end{enumerate}
\end{lem}

\begin{lem}\label{lem4}
     Let $\varepsilon\in Y_{h}$ satisfy~\eqref{epsilon}, $\rho^{h}\in \mathring{V}_{h}$ and $\mathbb{E}\in S_{\rho,h}$. Given $i,j\in\{0,1,\ldots,N-1\}$, $\eta_{ij}$ is given by~\eqref{eta}. Then $\mathbb{E}$ is curl free, i.e., $\nabla_{h}\times \mathbb{E}=0$ on $h(\mathbb{Z}+1/2)^{2}$, if and only if $\eta_{ij}=0$ for all $i$ and $j$.
\end{lem}

Lemma~\ref{lem3} gives some properties of local relaxation methods, which will be used to prove the convergence of the algorithms proposed in Section~\ref{algos}.
\begin{lem}\label{lem3}
     Let $\varepsilon\in Y_{h}$ satisfy~\eqref{epsilon}, $\rho^{h}\in \mathring{V}_{h}$ and $\mathbb{E}\in S_{\rho,h}$.
    \begin{enumerate}
        \item Given $i,j\in\{0,1,\ldots,N-1\}$. Let $\hat{\mathbb{E}}$ be updated from $\mathbb{E}$ by~\eqref{change0} with $\eta_{i,j}$ given by~\eqref{eta}. Then $\hat{\mathbb{E}}\in S_{\rho,h}$ and $\mathscr{A}_{h}(\varepsilon\hat{\mathbb{E}})=\mathscr{A}_{h}(\varepsilon \mathbb{E})$.
        \item Given $i_{\min}, j_{\min}, i_{\max}, j_{\max} \in\{0,1,\ldots,N-1\}$ with $i_{\min}<i_{\max}$ and $j_{\min}<j_{\max}$. Let $\hat{\mathbb{E}}$ be updated from $\mathbb{E}$ by~\eqref{changecu} with $\eta$ given by~\eqref{etacu}. Then $\hat{\mathbb{E}}\in S_{\rho,h}$ and $\mathscr{A}_{h}(\varepsilon\hat{\mathbb{E}})=\mathscr{A}_{h}(\varepsilon \mathbb{E})$.
        \item Let $\hat{\mathbb{E}}$ be updated from $\mathbb{E}$ by~\eqref{updateetax} with $\eta_{x}$ given by~\eqref{eta_x} for all $j\in\{0,1,\ldots,N-1\}$ and ~\eqref{updateetay} with $\eta_{y}$ given by~\eqref{eta_y} for all $i\in\{0,1,\ldots,N-1\}$. Then $\hat{\mathbb{E}}\in S_{\rho,h}$ and $\mathscr{A}_{h}(\hat{\mathbb{E}})=0$.
    \end{enumerate}
\end{lem}

\begin{proof}
    The first two items can be verified by direct substitution. The proof of the third item is given below. Since the updates~\eqref{updateetax}  for different $j$ are independent, given $j\in\{0,1,\ldots,N-1\}$ fixed,  
    \begin{align*}
        \sum_{i=0}^{N-1}\hat{E}_{i+\frac{1}{2},j}&=\sum_{i=0}^{N-1}E_{i+\frac{1}{2},j}+\frac{\eta_x}{\varepsilon_{i+\frac{1}{2},j}}\\
        &=\sum_{i=0}^{N-1}E_{i+\frac{1}{2},j}+\eta_x\sum_{i=0}^{N-1}\frac{1}{\varepsilon_{i+\frac{1}{2},j}}\\
        &=\sum_{i=0}^{N-1}E_{i+\frac{1}{2},j}+\left(-\sum_{i=0}^{N_x-1} E_{i+\frac{1}{2}, j}\bigg{/}\sum_{i=0}^{N_x-1} \frac{1}{\varepsilon_{i+\frac{1}{2}, j}}\right)\sum_{i=0}^{N-1}\frac{1}{\varepsilon_{i+\frac{1}{2},j}}\\
        &=0.
    \end{align*}
Therefore,
\begin{equation*}
    \sum_{j=0}^{N-1}{\sum_{i=0}^{N-1}\hat{E}_{i+\frac{1}{2},j}}=0.
\end{equation*}
Similarly, based on the update~\eqref{updateetay}, it follows that
\begin{equation*}
    \sum_{i=0}^{N-1}{\sum_{j=0}^{N-1}\hat{E}_{i,j+\frac{1}{2}}}=0.
\end{equation*}
Thus, we get $\mathscr{A}_{h}(\hat{\mathbb{E}})=0$ directly.
\end{proof}

By Lemma~\ref{lem2}, $\mathscr{A}_{h}(\mathbb{E}_{min}^{h})=0$ is necessary for $\mathbb{E}_{min}^{h}$ to be a minimizer of the functional $\mathcal{F}^{h}$. Moreover, Lemma~\ref{lem3} shows that the line shift acts as the mechanism that enforces the total zero field constraint on the discrete electric field. Without this step, the iterative scheme cannot ensure that the limiting electric field satisfies the total zero field condition. Therefore, the line shift is essential.
We now present the convergence analysis of the algorithm in Theorem~\ref{th2}.

\begin{thm}\label{th2}
    Let $\varepsilon\in Y_{h}$ satisfy~\eqref{epsilon}, $\rho^{h}\in \mathring{V}_{h}$  and $\mathbb{E}_{min}^{h}\in S_{\rho,h}$ be given by the unique minimizer of $F^{h}:S_{\rho,h}\rightarrow\mathbb{R}$. Let $\mathbb{E}^{(0)}\in\ S_{\rho,h}$ and $\mathbb{E}^{(t)}$ be the sequence generated by Algorithm~\ref{alg:multigrid}.
    \begin{enumerate}
        \item If the sequence is finite ending at $\mathbb{E}^{(m)}$, then $\mathbb{E}^{(m)}=\mathbb{E}_{min}^{h}$.
        \item If the sequence is infinte, then $\mathbb{E}^{(t)}\rightarrow \mathbb{E}^{h}_{min}$ on $h(\mathbb{Z}+1/2)^{2}$.
    \end{enumerate}
\end{thm}

\begin{proof}
    1. If the sequence is finite ending at $\mathbb{E}^{(m)}$, then $\eta_{i,j}=0$ for all $i,j\in\{0,1,\ldots,N-1\}$. Thus , $\mathbb{E}^{(m)}$ is curl-free by Lemma~\ref{lem4} and $\mathscr{A}_{h}(\mathbb{E}^{(m)})=0$ by Lemma~\ref{lem3}. Therefore, we can get $\mathbb{E}^{(m)}=\mathbb{E}_{min}^{h}$ from Lemma~\ref{lem2}.

2. Given $i,j\in\{0,1,\ldots,N-1\}$, denote $\eta^{(t)}_{i,j}$ by~\eqref{eta} for $\mathbb{E}=\mathbb{E}^{(t)}$. We claim that 
\begin{equation}\label{thc1}
\lim_{t\rightarrow\infty}\eta_{i,j}^{(t)}=0 \quad \forall i,j\in\{0,1,\ldots N-1\}.
\end{equation}

Suppose that~\eqref{thc1} is true. We prove that $\mathbb{E}^{(t)}\rightarrow \mathbb{E}_{min}^{h}$. It suffices to prove that for any convergent subsequence $\mathbb{E}^{(t_r)}$ of $\mathbb{E}^{(t)}$ with $\mathbb{E}^{(t_r)}\rightarrow \mathbb{E}^{(\infty)}$, then $\mathbb{E}^{(\infty)}=\mathbb{E}_{min}^{h}$.  Since function $\mathscr{A}_{h}:Y_{h}\rightarrow\mathbb{R}$ depends continuously on $\mathbb{E}$, we have $\mathscr{A}_{h}(\mathbb{E}^{(\infty)})=0$. And with the assumption~\eqref{thc1}, $\mathbb{E}^{(\infty)}\in S_{\rho,h}$ and $\eta_{i,j}^{(\infty)}=0$. Hence $\mathbb{E}^{(\infty)}$ is curl-free by Lemma~\ref{lem4}. Thus, $\mathbb{E}^{(\infty)}=\mathbb{E}_{min}^{h}$ by Lemma~\ref{lem2}.

We now prove the claim~\eqref{thc1}. Note that the iteration from $\mathbb{E}^{(t)}$ to $\mathbb{E}^{(t+1)}$ consists of many cycles of hierarchical mesh updates, single mesh updates and line shift updates. For convenience, we redefine the sequence of updates. Still start from $\mathbb{E}^{(0)}$, $\mathbb{E}^{(t+1)}$ is obtain form $\mathbb{E}^{(t)}$ by a hierarchical mesh update~\eqref{changecu}, a single mesh update~\eqref{change0} and a line shift update~\eqref{updateetax} or~\eqref{updateetay}. Since the original sequence is a subsequence of the new one, we only need to prove claim~\eqref{thc1} is true for the new sequence.

Since $\mathcal{F}^{h}[\mathbb{E}^{(t)}]>0$ decreases as $t$ increases, the limit $\mathcal{F}^{h}_{\infty}:=\lim_{t\rightarrow\infty}\mathcal{F}^{h}[\mathbb{E}^{(t)}]$ exists. And the change between two steps has the following property,
\begin{equation}\label{deltat}
\lim_{t\rightarrow\infty}\delta_{t}:=\lim_{t\rightarrow\infty}(\mathcal{F}^{h}[\mathbb{E}^{(t)}]-\mathcal{F}^{h}[\mathbb{E}^{(t+1)}])=0.
\end{equation}
We still denote $\eta^{(t)}_{i,j}$ by~\eqref{eta} for $\mathbb{E}=\mathbb{E}^{(t)}$. Let $\tau<t$ be the time of the last single mesh update at $(i,j)$ and $\eta^{(\tau)}_{i,j}$ denotes the corresponding value at that update, thus we have $\eta_{i,j}^{(\tau+1)}=0$. Notice that $\eta^{(t)}_{i,j}$ is a linear combination of $E^{(t)}_{i+\frac{1}{2},j},E^{(t)}_{i+\frac{1}{2},j+1},E^{(t)}_{i,j+\frac{1}{2}}$ and $E^{(t)}_{i+1,j+\frac{1}{2}}$, whose values is obtained from some previous updates. The update that determines all these values has three cases: the first case is a single  mesh update~\eqref{change0} for some $(m,n)$ in $\mathbb{E}^{(t')}$ with $\tau<t'< t$, the second one is a hierarchical mesh update~\eqref{changecu} for some $(i_{\min},j_{\min},i_{\max},j_{\max})$ in $\mathbb{E}^{(t')}$ with $\tau<t'< t$, the last one is line shift update~\eqref{updateetax}
or~\eqref{updateetay} in $\mathbb{E}^{(t')}$ with $\tau<t'< t$.

Consider the first case. Denote the perturbation is $\eta_{m,n}^{(t')}$, by assumption~\eqref{epsilon}, it follows that
\begin{equation*}
    \frac{4}{\varepsilon_{max}^{2}}[\eta_{m,n}^{(t')}]^2 \leq||\mathbb{E}^{(t'+1)}-\mathbb{E}^{(t')}||_{h}^2\leq \frac{4}{\varepsilon_{min}^{2}}[\eta_{m,n}^{(t')}]^2,
\end{equation*}
and by~\eqref{deltat} and we have the fact that $t'\rightarrow\infty$ as $t\rightarrow\infty$,
\begin{equation}\label{EE}   
\lim_{t\rightarrow\infty}||\mathbb{E}^{(t'+1)}-\mathbb{E}^{(t')}||_{h}^2=\lim_{t\rightarrow\infty}[\eta_{m,n}^{(t')}]^2=0.
\end{equation}
It follows from equation~\eqref{change0} that
\begin{equation}\label{uper}
    |E^{(t'+1)}_{i+\frac{1}{2},j}-E^{(t')}_{i+\frac{1}{2},j}|\leq\frac{\eta_{m,n}^{(t')}}{\varepsilon_{\min}h},
\end{equation}
and the variation of 
$E^{(t')}_{i+\frac{1}{2},j+1},E^{(t')}_{i,j+\frac{1}{2}}$ and $E^{(t')}_{i+1,j+\frac{1}{2}}$ admit the same upper bound as that in~\eqref{uper}. Since the single part update is calculated by~\eqref{eta}, we have
\begin{equation}
    |\eta_{i,j}^{(t'+1)}-\eta_{i,j}^{(t')}|\leq\frac{\varepsilon_{\max}}{\varepsilon_{\min}}\eta_{m,n}^{(t')}.
\end{equation}
Combining with~\eqref{EE}, we obtain $\eta_{i,j}^{(t'+1)}-\eta_{i,j}^{(t')}\rightarrow 0$ as $t\rightarrow \infty$. 

Next consider the second case. Denote the perturbation is $\tilde{\eta}^{(t')}$. Similarly, we have
\begin{equation}\label{EEm}
\frac{2(i_{\max}+j_{\max}-i_{\min}-j_{\min})}{\varepsilon_{\max}^{2}}\leq\frac{||\mathbb{E}^{(t'+1)}-\mathbb{E}^{(t')}||_{h}^2}{[\tilde{\eta}^{(t')}]^2}\leq\frac{2(i_{\max}+j_{\max}-i_{\min}-j_{\min})}{\varepsilon_{\min}^2},
\end{equation}
and thus 
\begin{equation*}  
\lim_{t\rightarrow\infty}||\mathbb{E}^{(t'+1)}-\mathbb{E}^{(t')}||_{h}^2=\lim_{t\rightarrow\infty}[\tilde{\eta}^{(t')}]^2=0.
\end{equation*}
Similar to the proof of Case 1, it holds that $\eta_{i,j}^{(t'+1)}-\eta_{i,j}^{(t')}\rightarrow 0$ as $t\rightarrow \infty$.

Last, consider Case 3. We still have $||\mathbb{E}^{(t'+1)}-\mathbb{E}^{(t')}||_{h}^2\rightarrow 0$ and $\eta_{x}^{t'}\rightarrow 0$ ($\eta_{y}^{t'}\rightarrow 0$) as $t\rightarrow \infty$. Completely analogous to Cases 1 and 2, $\eta_{i,j}^{(t'+1)}-\eta_{i,j}^{(t')}\rightarrow 0$ as $t\rightarrow \infty$ holds.

Since only finitely many updates occur between 
$\tau$ and $t$, thus $\eta_{i,j}^{(t)}\rightarrow\eta_{i,j}^{(\tau+1)} =0$ as $t\rightarrow \infty$. 
At this point, the convergence of the algorithm has been established. 
\end{proof}

\section{Numerical results}\label{NumResults}
In this section, we conduct numerical experiments to verify the accuracy and effectiveness of the proposed algorithm. We focus on comparing the newly proposed HLR methods with the original single mesh method from multiple perspectives, including metrics like accuracy, number of iteration steps and computational time. 

\subsection{Accuracy and comparison for Poisson's equation with exact solution} 
 Consider $\Omega=(0,4)\times(0,4)$ with periodic boundary conditions. The Poisson's equation has the analytical solution for inhomogeneous dielectric  permittivity
 \begin{equation}\label{exactphi}
       \phi(x,y)=\cos\left(\frac{\pi x}{2}\right)\sin\left(\frac{\pi y}{2}\right),\quad \text{and}\quad \varepsilon(x,y)=2+\cos\left(\frac{\pi x}{2}\right)\cos\left(\frac{\pi y}{2}\right),
 \end{equation}
where the source term $\rho$ is analytically computed according to Poisson’s equation $-\nabla\cdot(\varepsilon\nabla\phi)=\rho$  such that the unique solution is given by the first equation of~\eqref{exactphi}. 

In our calculations, the Poisson's equation is solved by local algorithms via three distinct relaxation methods over a series of computational grids. Based on the given charge density distribution, the electric field is first initialized to satisfy Gauss’s law by the same procedure. Upon convergence of the relaxation iteration, numerical $L^\infty$ errors of the electric fields are calculated. As shown in Table~\ref{errta}, both the forward and zigzag HLR methods achieve second-order convergence with respect to the number of grid points $N_x=N_y=N$ in terms of $L^\infty$ error, a convergence behavior consistent with that of the single mesh update method. Notably, the derived convergence characteristics align with the conclusions reported in \cite{li2024finite}.

\begin{table}[H]
\centering
\begin{tabular}{c c c c c c c}

\hline
$N$& 
Single& Order & Forward & Order&
Zigzag& Order\\
\hline
32  & 8.157469E-3 & - & 8.157469E-3 & - & 8.157469E-3 & -\\

  64& 2.051296E-3 & 1.9916 & 2.051296E-3 & 1.9916 &2.051296E-3& 1.9916 \\

128  & 5.135733E-4 & 1.9979 & 5.135728E-4 & 1.9979 & 5.135728E-4 & 1.9979\\

256  & 1.284438E-4 & 1.9994 & 1.284400E-4 & 1.9994 &1.284400E-4&1.9994 \\
\hline
\end{tabular}
\caption{The $L^\infty$ error and convergence order for three relaxation methods.}
\label{errta}
\end{table}

To elucidate the convergence behavior of the curl residual, we analyze its spatial profile along the $x$-direction at the cross-section $y=0.5$. The residual distributions at the initial stage and after 50 and 100 iterations for three relaxation methods with $N=128$ are compared in Fig.~\ref{fequency}. While the single mesh relaxation rapidly suppresses high-frequency residuals, it converges slowly for low-frequency components. In contrast, the HLR methods effectively reduce residuals over a broad range of spatial scales, leading to faster decay of low-frequency error modes. Moreover, the zigzag HLR method consistently shows stronger attenuation of low-frequency modes than the forward HLR method.
        
       

\begin{figure}[ht]
    \centering 
    \captionsetup[sub]{labelformat=empty}
    
    \begin{subfigure}{0.31\textwidth} 
        \centering
        \includegraphics[width=\linewidth]{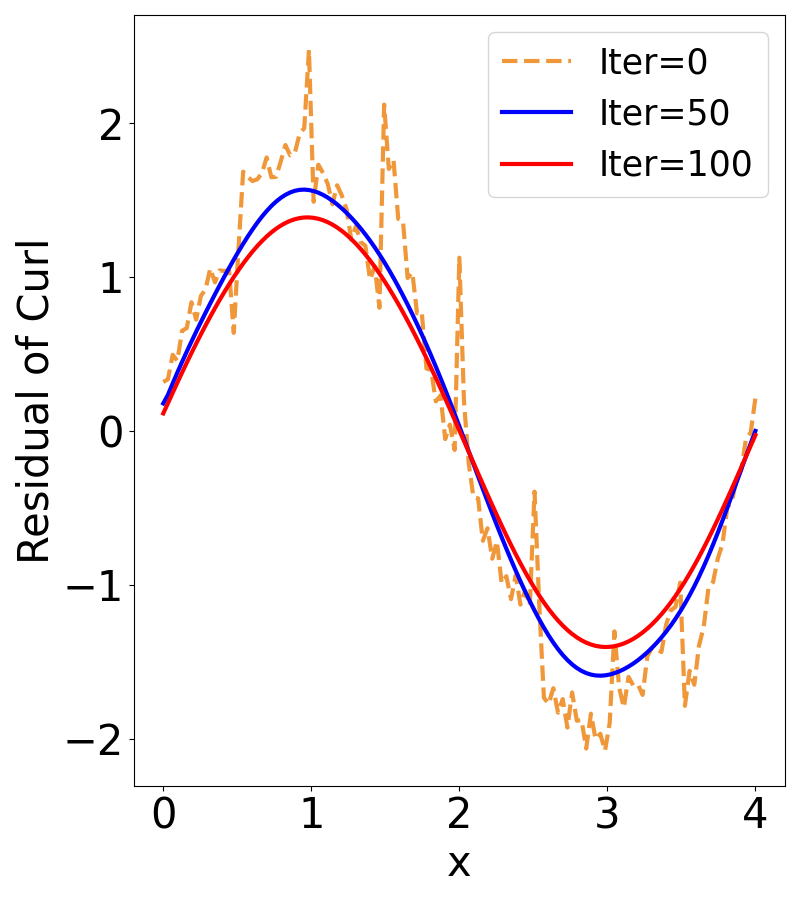} 
        \caption{(a) Single mesh} 
    \end{subfigure}    
    \begin{subfigure}{0.31\textwidth}
        \centering
        \includegraphics[width=\linewidth]{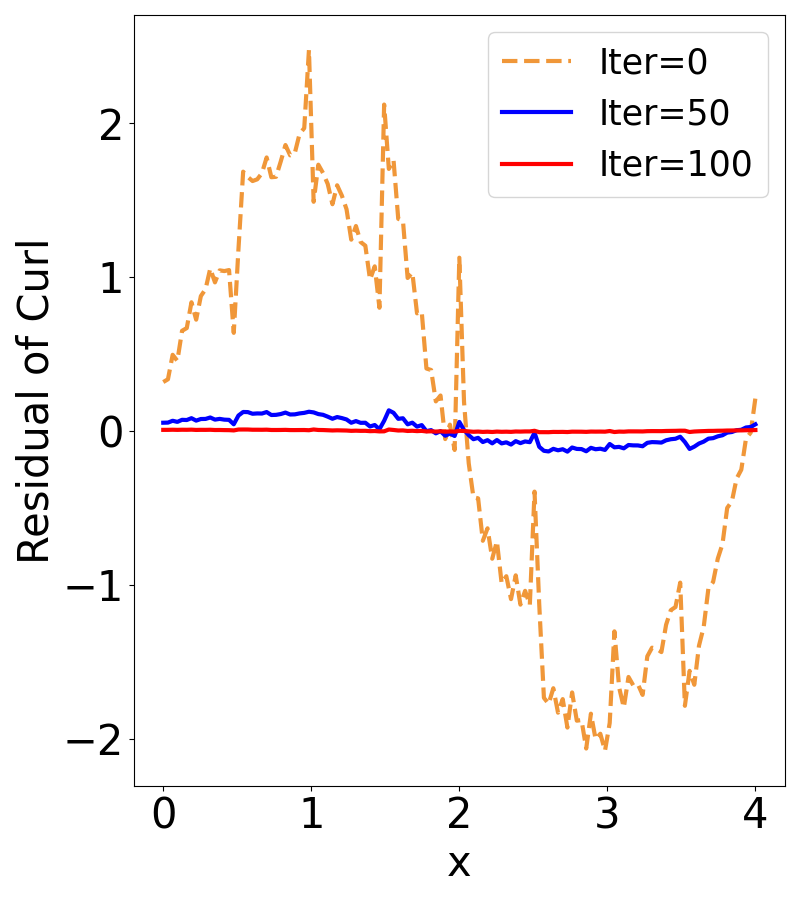}
        \caption{(b) Forward HLR}       
    \end{subfigure}
    \begin{subfigure}{0.31\textwidth}
        \centering
        \includegraphics[width=\linewidth]{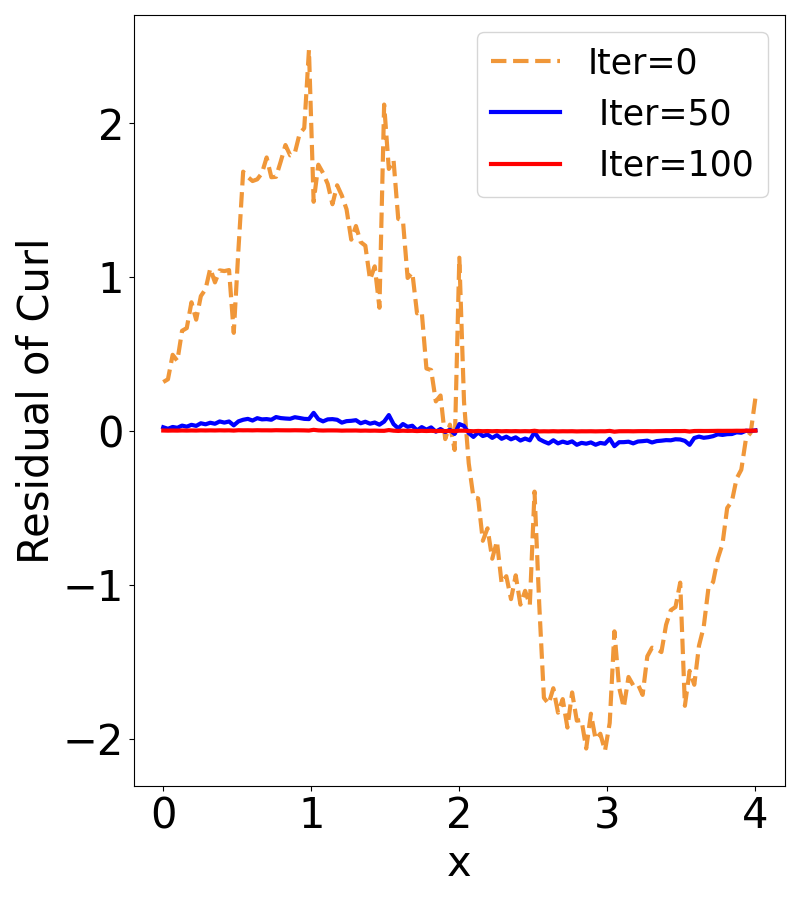}
        \caption{(c) Zigzag HLR}       
    \end{subfigure}    
    
    \caption{Distribution of the curl residual along the $x$-direction at the cross-section $y = 0.5$ for three relaxation methods. Panels (a)–(c) correspond to single mesh relaxation, forward HLR and zigzag HLR methods, respectively.}
\label{fequency}
\end{figure}

\subsection{Time-dependent Poisson's equation tests}
In most electrolyte transport problems, charge concentrations evolve dynamically over time. This temporal variation necessitates the frequent solution of the Poisson's equation at every time step to update the electric potential field. To focus exclusively on the numerical solution of the Poisson's equation itself, we simplify the problem by neglecting the detailed charge transport dynamics and instead abstract the original time-dependent system into a sequence of independent Poisson's equation solves. We then apply the three aforementioned relaxation schemes to carry out these numerical solutions and compare their performance.


Denote $\Omega=(0,4)\times(0,4)$ with periodic boundary conditions. Let the initial charge density distribution and the homogeneous dielectric permittivity  be given by 
\begin{equation*}
  \rho(x,y)=0,\quad \text{and}\quad \varepsilon(x,y)=1.
\end{equation*} 
Denote the charge density  at the $n$ th time step by $\rho^{(n)}$, which is iteratively given by
\begin{equation}\label{deltarho}
    \rho^{(n+1)}=\rho^{(n)}+\frac{1}{M^{(n)}}\sum_{k=1}^{16}\left[a_{k}^{(n)}\cos\left(\frac{k\pi x}{2}\right)\sin\left(\frac{k\pi y}{2}\right)+b_{k}^{(n)}\sin\left(\frac{k\pi x}{2}\right)\cos\left(\frac{k\pi y}{2}\right)\right].
\end{equation}
for $M^{(n)}=64\sum_{k=1}^{16}(a_{k}^{(n)}+b_{k}^{(n)})$, and $a_{k}^{(n)},\ b_{k}^{(n)}$ being random variables uniformly distributed on the interval $(0,1)$.


Fig.~\ref{rhotime} shows the average computational time per time step as a function of the grid points $N$ for three relaxation methods and an FFT-based solver, averaged over 100 time steps. At each time step, the iterations are terminated when the tolerance $\varepsilon_{tol}$ reaches $10^{-7}$ (left panel) and $10^{-9}$ (right panel). The results indicate that the HLR methods lead to a clear reduction in computational cost compared to the single mesh method. For $\varepsilon_{tol}=10^{-7}$
, HLR methods are still competitive with the FFT-based solver, requiring roughly eight times the computational time over the range of spatial step sizes considered. Owing to the weak dependence of the iteration count on grid refinement, together with the hierarchical structure of the algorithm, the overall computational complexity is approximately $\mathcal{O}(N\text{log}N)$. 

In the case of spatially varying dielectric permittivity, we set the initial charge density and permittivity are
\begin{equation*}
  \rho(x,y)=0,\quad \text{and}\quad \varepsilon(x,y)=2+\cos(\frac{\pi x}{2})\cos(\frac{\pi y}{2}).
\end{equation*} 
The charge density evolves as time according to the same rule as given in~\eqref{deltarho}. Fig.~\ref{rhotimevary} presents the average computational time per time step as a function of the grid points $N$ for the three relaxation methods in the case of inhomogeneous dielectric permittivity. As FFT-based methods are not applicable to inhomogeneous media, the computational times for the homogeneous permittivity case are plotted as dashed lines for comparison. It can be seen that HLR methods continue to provide a substantial reduction in computational cost and maintain FFT-like computational complexity. Moreover, compared with the homogeneous dielectric permittivity case, no additional computational complexity is incurred.

\begin{figure}[H]
  \centering 
    \captionsetup[sub]{labelformat=empty}
    \begin{subfigure}{0.48\textwidth} 
        \centering
        \includegraphics[width=\linewidth]{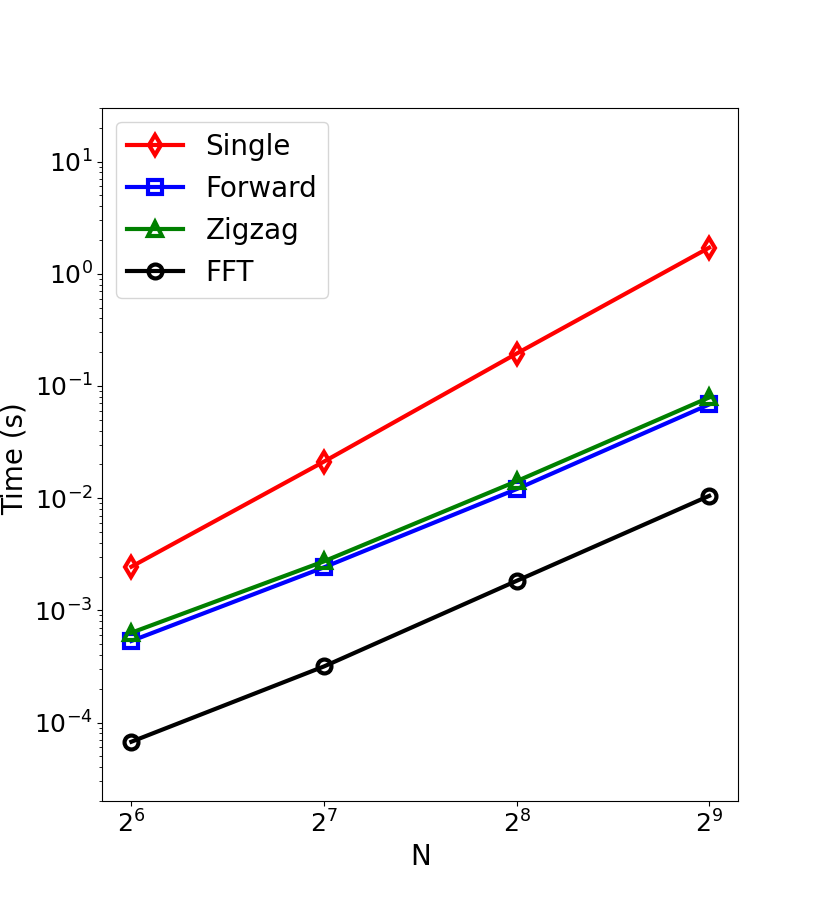} 
    \end{subfigure}
    \begin{subfigure}{0.48\textwidth}
        \centering
        \includegraphics[width=\linewidth]{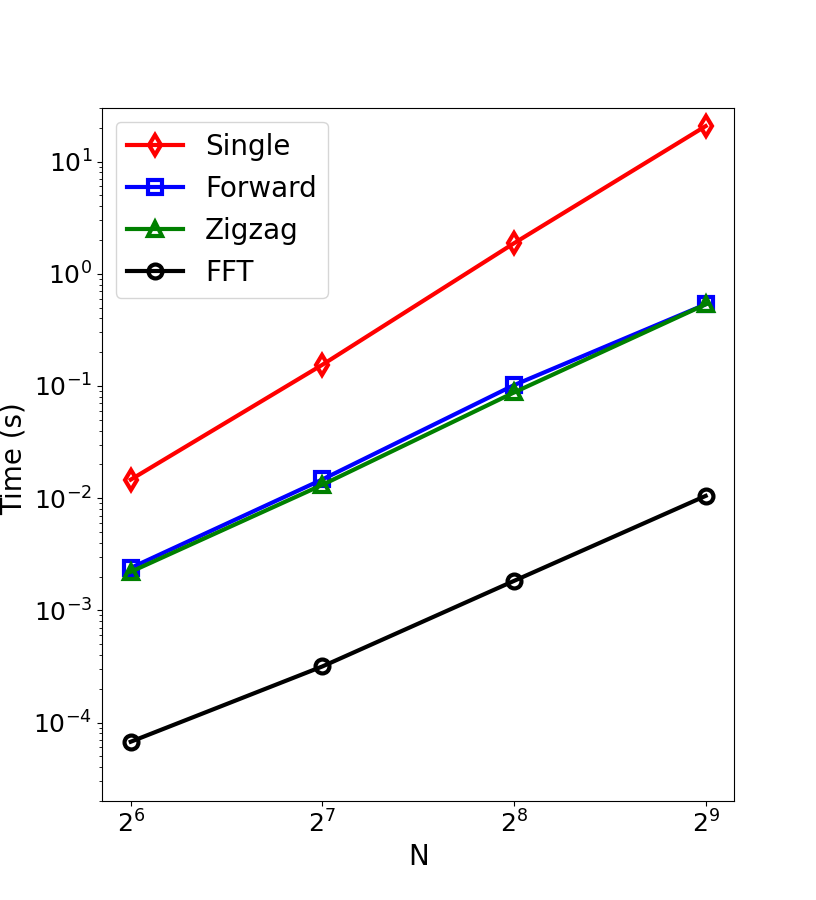}
    \end{subfigure}
  \caption{Average iteration time per time step versus the number of grid points $N$ for three relaxation methods and an FFT-based method, averaged over 100 time steps for homogeneous dielectric permittivity case. The stopping tolerance at each time step is set to:  $\varepsilon_{tol}=10^{-7}$ (left),  $\varepsilon_{tol}=10^{-9}$ (right).}
  \label{rhotime}
\end{figure}

\begin{figure}[H]
  \centering 
    \captionsetup[sub]{labelformat=empty}
    \begin{subfigure}{0.48\textwidth} 
        \centering
        \includegraphics[width=\linewidth]{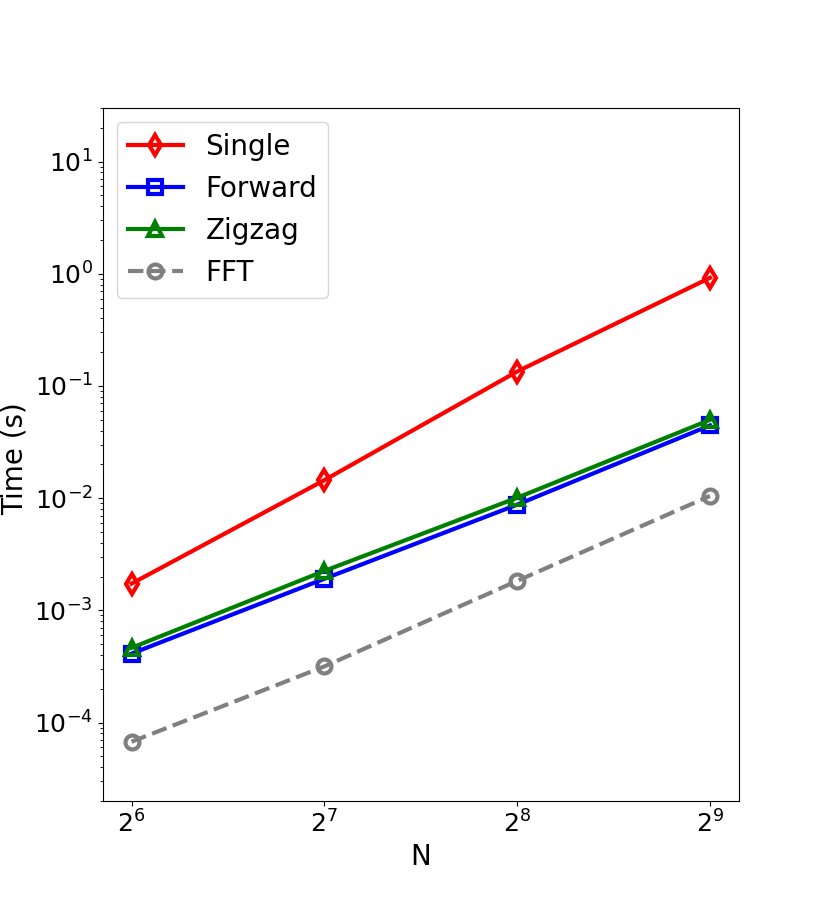} 
    \end{subfigure}
    \begin{subfigure}{0.48\textwidth}
        \centering
        \includegraphics[width=\linewidth]{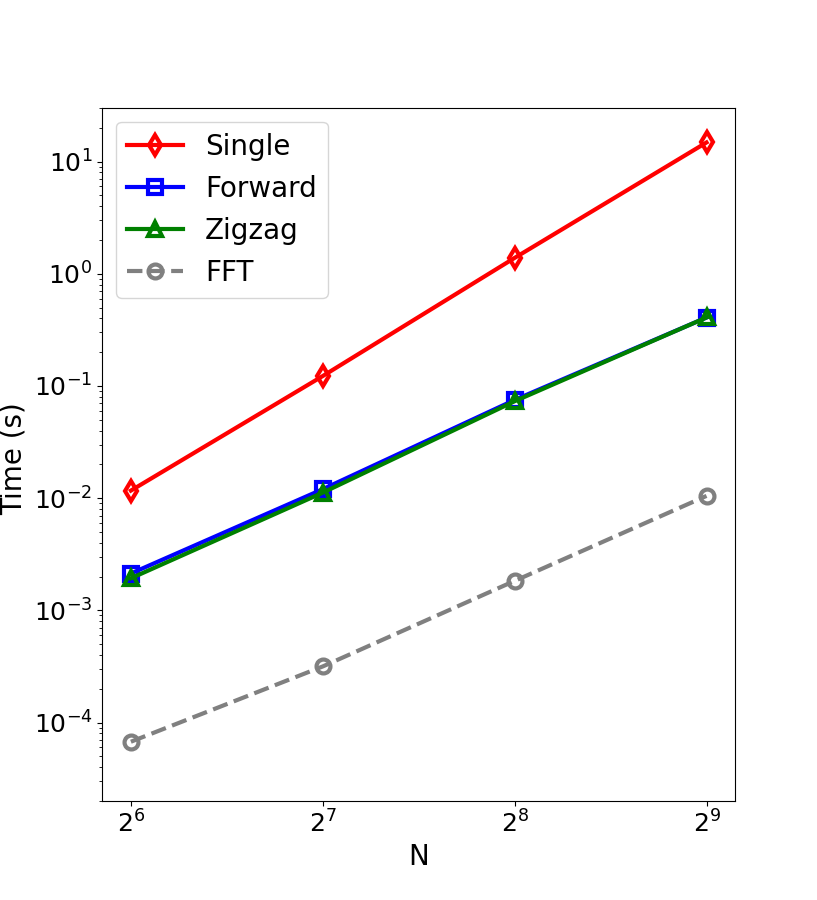}
    \end{subfigure}
  \caption{Average iteration time per time step versus the number of grid points $N$ for three relaxation methods averaged over 100 time steps for inhomogeneous dielectric permittivity case. The stopping tolerance at each time step is set to:  $\varepsilon_{tol}=10^{-7}$ (left),  $\varepsilon_{tol}=10^{-9}$ (right). The dash lines represent the FFT results of Figure \ref{rhotime}.}
  \label{rhotimevary}
\end{figure}

\subsection{Poisson–Boltzmann equation tests}
The Poisson–Boltzmann (PB) equation governs electrostatic interactions at thermodynamic equilibrium~\cite{ Gouy:JP:1910,honig1995classical,tu2022linear,xu2014solving}. In the presence of fixed charge density and spatially varying dielectric coefficient, the dimensionless PB equation can be written as
\begin{equation*}
    -\nabla \cdot  \varepsilon \nabla \phi = \kappa^2\left[\rho^f+\sum_{\ell} z^\ell\frac{N^{\ell}\exp(-z^\ell\phi)}{\int_{\Omega}\exp(-z^\ell\phi)d\bm x}\right].
\end{equation*}
where $\kappa$ is the dimensionless parameter, $\rho^f$ is fixed charge density, $N^{\ell}$ is the total number of the $\ell$-th charged particle and $z^\ell$ is its valence. Notably, solving the PB equation can be reformulated as a constrained minimization problem. The associated energy functional admits minimization via a Maggs-type algorithm~\cite{ BSD:PRE:2009,ZWL:PRE:2011}, yielding a numerical scheme that is not only straightforward and easy to implement, but also unconditionally stable. For a system consisting of 
$M$ types of particles, it admits minimizing the following energy functional subject to the constraints of discrete Gauss’s law and mass conservation:
\begin{equation*}
    \mathcal{F}_{PB} = \int_\Omega \left( \frac{\varepsilon}{2} |\bm E|^2 + \sum_{\ell=1}^M c^\ell \ln c^\ell \right) d \bm x.
\end{equation*}
The Maggs-type algorithm for solving this minimization problem consists of two core iterative steps: field updates applied to all finest cell in the grid, and concentration updates executed for all pairs of adjacent movable grids~\cite{BSD:PRE:2009}. A key feature of this algorithm is that both update steps rigorously preserve the discrete Gauss’s law and mass conservation constraints, which are essential for the physical consistency of Poisson–Boltzmann (PB) equation solutions. The field update step is identical to that used for solving the standard Poisson's equation in this work, the only additional component required for the PB equation is the concentration update step. The HLR methods proposed in this paper are designed to accelerate the convergence of the Maggs-type solver, thereby improving the computational efficiency of PB equation solutions.  

This section presents three-dimensional ($3\text{D}$) numerical examples to validate the proposed methods.
In the numerical example, we denote $\lambda_0$ as the fundamental length scale, corresponding to the Lennard–Jones diameter used in the simulation. A spherical colloidal particle with radius 
$r=3\lambda_0$ is placed at the center of a cubic simulation box $(-\frac{L}{2},\frac{L}{2})\times(-\frac{L}{2},\frac{L}{2})\times(-\frac{L}{2},\frac{L}{2})$ with periodic boundary conditions, where $L=30\lambda_0$. The colloidal particle carries a total charge 
$Ze=60e$, which is uniformly distributed over its surface.
The solvent contains only one ionic species, namely monovalent counterions with valence $z^1=-1$. To maintain overall charge neutrality, counterions with a total charge of $-Ze$ are initially uniformly distributed in the region outside the colloidal particle. The system is discretized on three-dimensional grids with different spatial resolutions, $N_x=N_y=N_z=32,~64$ and $128$. An iterative stopping criterion is imposed such that the energy change associated with the field update is below $\varepsilon_{tol}=10^{-7}$. The electrostatic energy functional is then minimized using the single mesh relaxation, forward HLR, and zigzag HLR methods to obtain the equilibrium counterion distribution.


Fig.~\ref{PB} presents plots of cross-sectional at the plane $z=0$ for the cunterion concentrations at equilibrium obtained using the HLR methods in different spatial resolutions. As expected, due to the positive surface charge on the colloidal sphere, counterion accumulate in its vicinity. Upon reaching equilibrium, the cunterion concentrations obtained from both the forward HLR and zigzag HLR are nearly identical to those obtained from the single mesh update in \cite{BSD:PRE:2009}.

\begin{figure}[ht]
    \centering 
    \captionsetup[sub]{labelformat=empty}
    \begin{subfigure}{0.3\textwidth} 
        \centering
        \includegraphics[width=\linewidth]{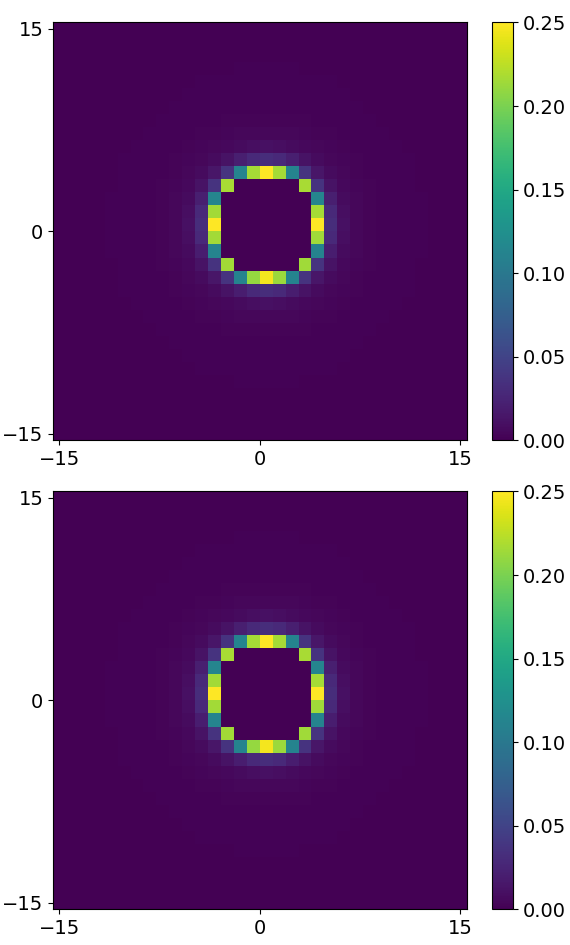} 
        \caption{(a) $32\times 32\times 32$} 
        
    \end{subfigure}
    \begin{subfigure}{0.3\textwidth}
        \centering
        \includegraphics[width=\linewidth]{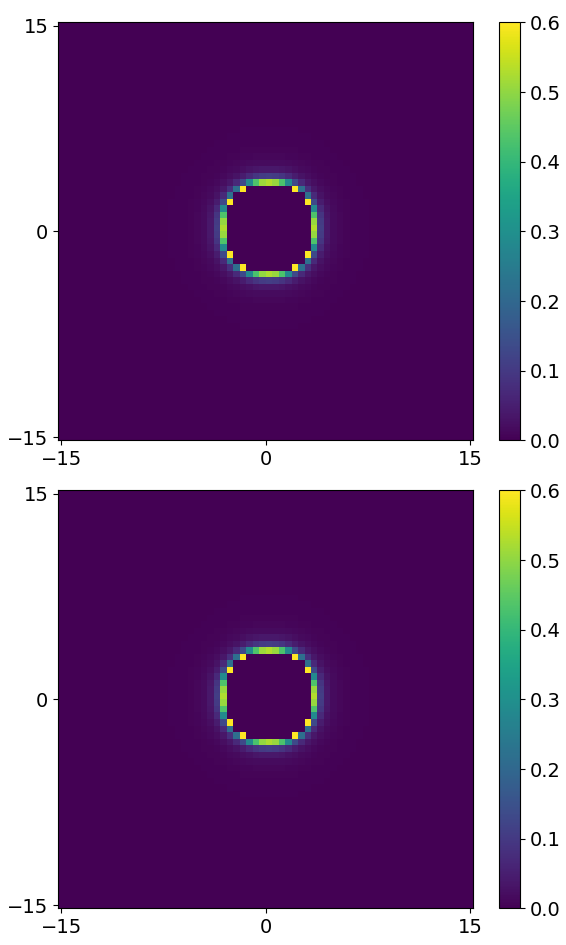}
        \caption{(b) $64\times 64\times 64$}
       
    \end{subfigure}
    \begin{subfigure}{0.3\textwidth}
        \centering
        \includegraphics[width=\linewidth]{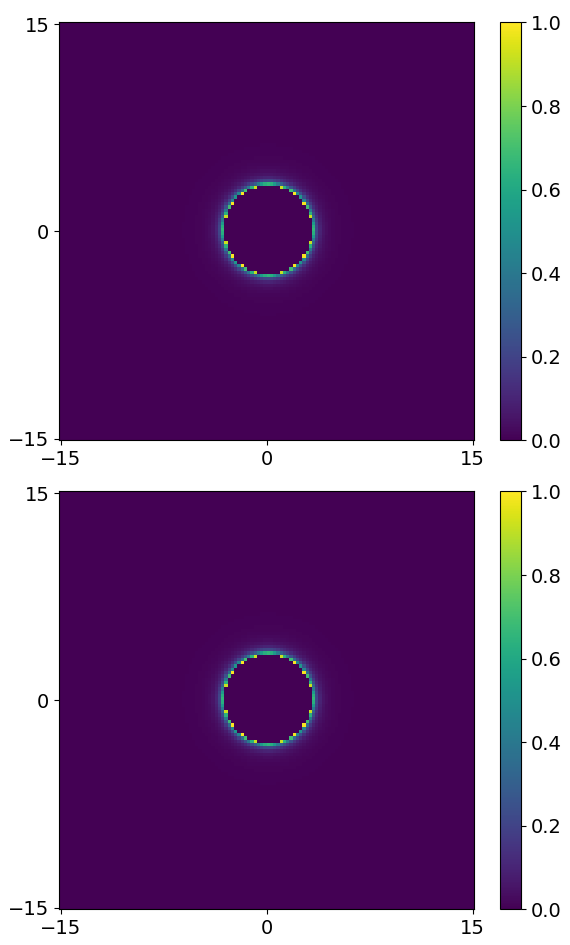}
        \caption{(c) $128\times 128\times 128$}
       
    \end{subfigure}
    \caption{Counterion concentration obtained using HLR methods in different discretized grid. Rows from top to bottom represent forward HLR and zigzag HLR method.}
\label{PB}
    \end{figure}



Table~\ref{pbt} presents the computational costs of the curl-free step and the total runtime (including concentration moves), for all three update methods upon reaching equilibrium. Compared with the single mesh method, both the forward HLR and zigzag HLR methods reduce the total computational time required to reach the same equilibrium distribution, with particularly noticeable improvements in the curl-free step. 
For low-resolution grids, the total number of iteration steps required for the zigzag HLR method to reach equilibrium is not significantly smaller than that for the forward HLR method. Moreover, a single local relaxation in the zigzag HLR method involves more computational operations than in the forward HLR method, resulting in a comparable overall computational time between the two methods. However, as the grid resolution increases, the zigzag HLR method requires noticeably fewer iteration steps to reach equilibrium than the forward HLR method. Consequently, despite the slightly higher cost of the curl-free relaxation in the zigzag scheme, the total computational time for the zigzag HLR method becomes lower than that of the forward HLR method.


\begin{table}[H]
\centering
\begin{tabular}{c c c c c c c c c c}
\hline
\multirow{2}{*}{N} & &\multicolumn{2}{c}{Single mesh} & &\multicolumn{2}{c}{ Forward HLR}& &\multicolumn{2}{c}{Zigzag HLR}\\
\cline{3-4}\cline{6-7}\cline{9-10}
& & Relaxation& Total & &  Relaxation& Total & & Relaxation& Total\\
\hline
32& & 13.3s & 35.1s & & 6.7s & 21.4s & & 7.5s & 21.3s\\
64& & 320.4s & 834.7s & & 171.9s & 530.4s & & 209.4s & 547.7s\\
128& & 6494.7s & 16822.0s & & 3767.3s & 11381.7s & & 4006.1s & 10262.0s\\
\hline

\end{tabular}
\caption{Computational time required by single mesh method and HLR methods for the total simulation and for the curl-free relaxation alone.}
\label{pbt}
\end{table}

\subsection{Poisson-Nernst-Planck equation simulations}\label{PNP}
The Poisson–Nernst–Planck (PNP) equations provide a time-dependent mean-field continuum framework for modeling ion transport in a wide range of systems \cite{CKC:BJ:03,daiguji:NL:2005,eisenberg2007poisson}, describing ionic transport in electrolyte solutions through the coupling of Poisson’s equation with the Nernst–Planck equations. However, the classical mean-field approximation neglects ion–ion correlation effects. To address this limitation, modified PNP models have been developed to incorporate both long-range Coulomb interactions and short-range hard-sphere (HS) correlations \cite{ma2021modified}. Furthermore, the PNP model can be reformulated within the Maxwell–Ampère Nernst–Planck (MANP) framework \cite{qiao2023maxwell} to enable charge transport modeling under a local algorithmic formulation. Within the MANP framework, the charge dynamics of a system with $M$ ionic species can be described by
\begin{equation*}
\left\{
\begin{aligned}
&\frac{\partial c^{\ell}}{\partial t} 
= \nabla \cdot \gamma^{\ell} \left( \nabla c^{\ell} - z^{\ell} c^{\ell} \bm E \right), 
\quad \ell = 1,2,\ldots,M, \\[6pt]
&\varepsilon\frac{\partial \bm E}{\partial t} 
= \sum_{\ell=1}^{M} z^{\ell} \gamma^{\ell} \left( \nabla c^{\ell} - z^{\ell} c^{\ell} {\bm E} \right) + \boldsymbol{\Theta}, \\[6pt]
&\nabla \cdot \boldsymbol{\Theta} = 0, \\[6pt]
&\nabla \times \bm E = \mathbf{0},
\end{aligned}
\right.
\end{equation*}
where $\gamma^{\ell}>0$ is the diffusion coefficient, $c^{\ell}$ is the concentration of the $\ell$-th particle and $z^{\ell}$ is its valence.  $\boldsymbol{\Theta}$ is introduced
as a degree of freedom to enforce $\nabla\times\bm E=0$, which implies the existence of
the electric potential satisfying the Poisson’s equation in a connected spatial domain.

Consider a representative numerical test case. All numerical experiments are performed using the same parameters and settings as in \cite{qiao2023maxwell}, with details provided in Appendix A.
The evolution process of particle concentration and the changes in the electric field during the evolution process has been studied. The snapshots of the cation concentration $c^1$  and electric field $|\bm E|$ at two time instants (T = 0.05 and 2) for three different local curl-free relaxation methods are shown in Fig.~\ref{1MANP}. It can be observed that with all three approaches, the anions are repelled around the negative fixed charges and subsequently diffuse to other areas. The similar evolution patterns observed for the three methods confirm that the HLR methods reproduce the expected behavior and are numerically consistent.
\begin{figure}[ht]
    \centering 
    \captionsetup[sub]{labelformat=empty}
    \begin{subfigure}{0.49\textwidth} 
        \centering
        \includegraphics[width=\linewidth]{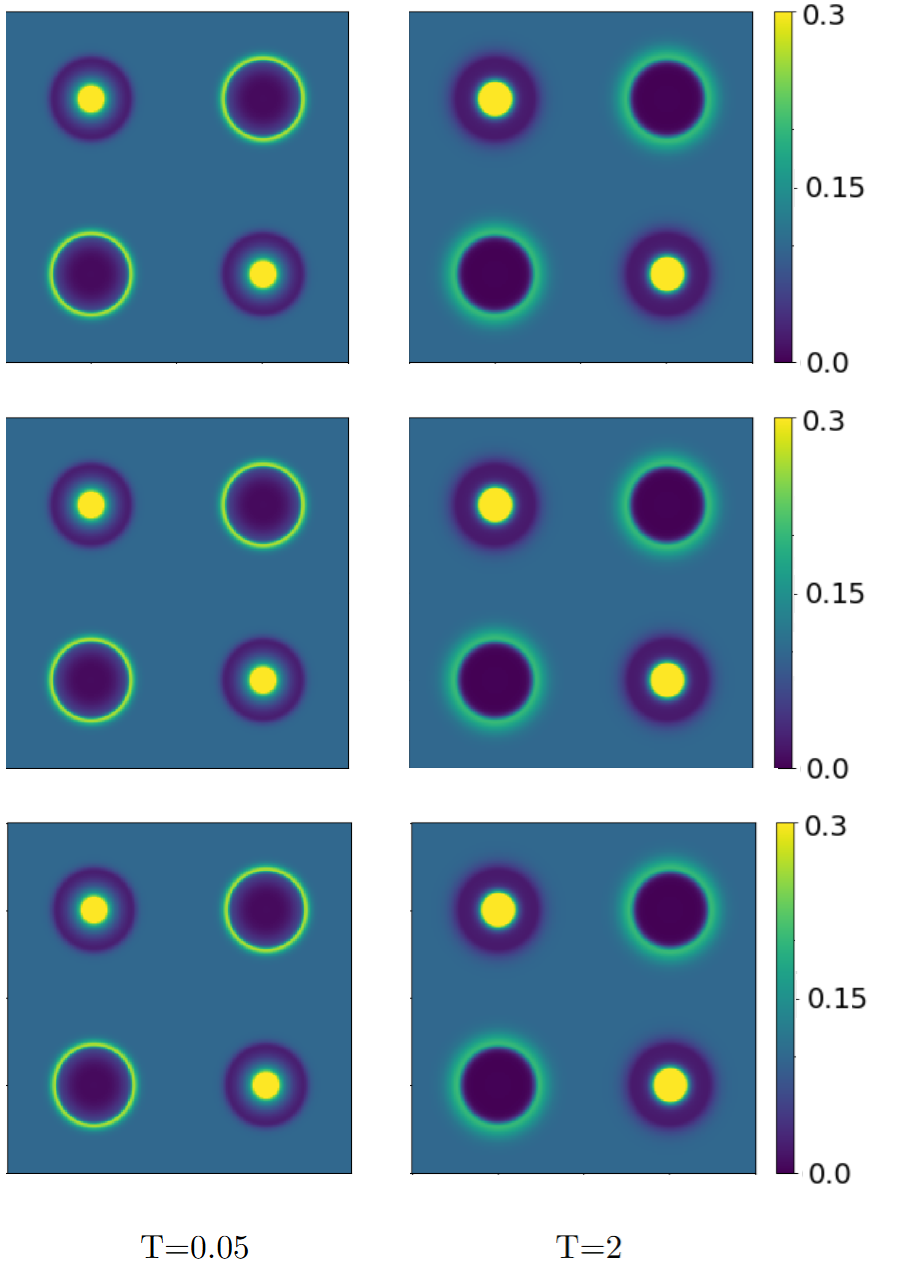} 
        \caption{(a)} 
        
    \end{subfigure}
    \begin{subfigure}{0.49\textwidth}
        \centering
        \includegraphics[width=\linewidth]{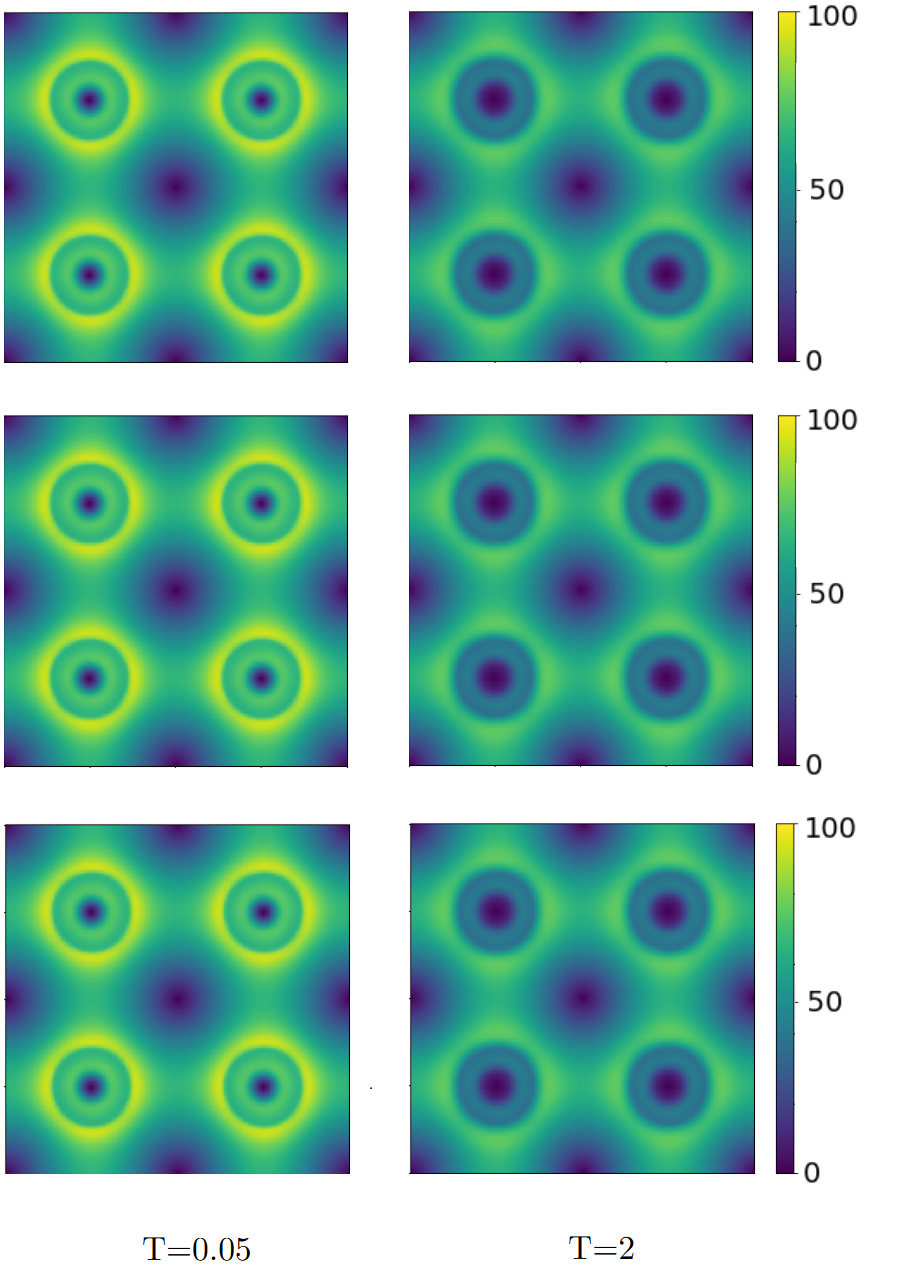}
        \caption{(b)}
       
    \end{subfigure}
    \caption{Snapshots of the cation concentration $c^1$ (a) and the magnitude of the electric field $|\bm E|$ (b) for different time. From top to bottom, the rows correspond to single mesh relaxation,  forward HLR and zigzag HLR, respectively.}
\label{1MANP}
    \end{figure}

The choice of $\boldsymbol{\Theta}$ and its impact on the efficiency of the local curl-free condition has been discussed in \cite{qiao2023maxwell}. Here, we adopt three specific forms of $\boldsymbol{\Theta}$, which is given by the following expressions,
\begin{align*}
\boldsymbol{\Theta}_{1}^{n} =& 0,\\
\boldsymbol{\Theta}_{2}^{n} =& 
\frac{\varepsilon\bm E^{n} - \varepsilon\bm E^{n-1}}{\Delta t} + \sum_{\ell=1}^{M} z^{\ell} \mathbf{J}^{\ell,n-1}
,\\
\boldsymbol{\Theta}_{3}^{n} =& \frac{3}{2} 
\left(
\frac{\varepsilon\bm E^{n} - \varepsilon\bm E^{n-1}}{\Delta t} + \sum_{\ell=1}^{M} z^{\ell} \mathbf{J}^{\ell,n-1}
\right)
- \frac{1}{2} 
\left(
\frac{\varepsilon\bm E^{n-1} - \varepsilon\bm E^{n-2}}{\Delta t} + \sum_{\ell=1}^{M} z^{\ell} J^{\ell,n-2}
\right),
\end{align*}
where $\mathbf{J}^{\ell, n} 
= -\gamma^{\ell} 
(
\nabla c^{\ell, n+1} 
- z^{\ell} c^{\ell, n+1} \bm E^{n}
)$. Let $N_x=N_y=256$, Table.~\ref{manp} presents the time of relaxation steps required to satisfy the stopping criterion $\varepsilon_{tol}=10^{-8}$ at each time step, for different choices of $\boldsymbol{\Theta}$ and update methods.


It can be observed that, for 
$\boldsymbol{\Theta}$ taking the value $\boldsymbol{\Theta}_1$ and $\boldsymbol{\Theta_3}$, the local algorithm combined with HLR methods requires only a few update rounds per time step, resulting in a significantly lower computational cost compared to the single mesh method. In these cases, the zigzag HLR method performs nearly the same number of iterations as the forward HLR method, and its higher per-iteration cost leads to the largest overall computational time. When $\boldsymbol{\Theta}=\boldsymbol{\Theta}_2$, the single mesh method requires multiple iterations only during the first few steps, after which only a single iteration per time step is needed. Similarly, the HLR methods require on average only one iteration per step. As a result, the total computational time of the single mesh method is comparable to that of the forward HLR method. Since the zigzag HLR method has a higher cost per local relaxation, it remains the most time consuming.

\begin{table}[H]
\centering
\begin{tabular}{c c c c c c c}
\hline
$\boldsymbol{\Theta}$ & &Single mesh& &  Forward HLR& &Zigzag HLR\\
\hline
$\boldsymbol{\Theta}_1$& & 0.0516s & & 0.0116s& & 0.0151s\\
$\boldsymbol{\Theta}_2$& & 0.0024s & & 0.0023s& & 0.0036s\\
$\boldsymbol{\Theta}_3$& & 0.0076s & & 0.0041s& & 0.0043s\\
\hline

\end{tabular}
\caption{Average computational time per time step for the forward HLR, zigzag HLR, and single mesh relaxation methods under different values of $\boldsymbol{\Theta}$.}
\label{manp}
\end{table}

\section{Conclusions}\label{s:Con} 

We develop a variational algorithm for solving Poisson’s equation, inspired by the local algorithm proposed by Maggs and its underlying dynamical properties. The algorithm consists of two stages. First, an initial electric field satisfying Gauss’s law is constructed, which may not be curl-free. Second, the electric field is iteratively updated through local relaxations to progressively reduce the electrostatic energy until convergence. The limited effectiveness of the single mesh relaxation in damping low-frequency curl modes motivates the development of two hierarchical schemes: the forward HLR and the zigzag HLR method. These hierarchical strategies significantly accelerate convergence by efficiently reducing low-frequency errors. We further establish a rigorous convergence proof for the proposed algorithm, and numerical experiments confirm the high efficiency of the multilevel relaxations. 
Beyond Poisson’s equation, the proposed multilevel curl-free relaxation schemes demonstrate improved convergence in a broader class of Maggs-based local algorithms, as evidenced by numerical tests on the Poisson–Boltzmann equation and by the application of the MANP framework to the Poisson–Nernst–Planck equation. Owing to the strictly local updates, the proposed algorithms are well suited for large-scale parallel implementations in particle-based simulations. Moreover, the variational and multilevel nature of the framework suggests promising extensions to coupled kinetic field systems, such as the Vlasov-Poisson equation, which will be explored in future work.

\section*{Acknowledgment}
Z. Xu and H. Zhou are financially supported by National Key R\&D Program of China (grant No. 2024YFA1012403), Natural Science
Foundation of China (grants No. {12325113} and
{12426304}) and SJTU Kunpeng \& Ascend Center
of Excellence. Q. Yin is partially supported by the Hong Kong Polytechnic University Postdoc Matching Fund Scheme 4-W425.

\section*{Data Availability}
Data will be made available on request.

\section*{Declarations}
\textbf{Competing interests} The authors have not disclosed any competing interests.

\appendix

\section{Numerical Parameters and Settings}

The benchmark in Section \ref{PNP} took the same setup as \cite{qiao2023maxwell}. The system is composed of binary ionic species. 
The fixed charge distribution is defined by the following equation, 
\begin{align*}
    \rho^{f}(x,y) = 
e^{-100\left[  \left( x + \frac{1}{2} \right)^{2} + \left( y + \frac{1}{2} \right)^{2} \right]}
- e^{-100\left[ \left( x + \frac{1}{2} \right)^{2} + \left( y - \frac{1}{2} \right)^{2} \right]} \qquad\\
- e^{-100\left[ \left( x - \frac{1}{2} \right)^{2} + \left( y + \frac{1}{2} \right)^{2} \right]}
+ e^{-100\left[ \left( x - \frac{1}{2} \right)^{2} + \left( y - \frac{1}{2} \right)^{2} \right]}.
\end{align*}
The dielectric permittivity is given by 
\[
\varepsilon(x,y) 
= \frac{\varepsilon_{\text{max}} - \varepsilon_{\text{min}}}{2} \Big[ \tanh(100d - 25) + 1 \Big] + \varepsilon_{\text{min}},
\]
where $d$ is
\[
d = 
\begin{cases}
\sqrt{ \left( x - \frac{1}{2} \right)^{2} + \left( y - \frac{1}{2} \right)^{2} }, 
& 0 < x \leq 1,\; 0 < y \leq 1, \\[6pt]
\sqrt{ \left( x - \frac{1}{2} \right)^{2} + \left( y + \frac{1}{2} \right)^{2} }, 
& 0 < x \leq 1,\; -1 \leq y \leq 0, \\[6pt]
\sqrt{ \left( x + \frac{1}{2} \right)^{2} + \left( y - \frac{1}{2} \right)^{2} }, 
& -1 \leq x \leq 0,\; 0 < y \leq 1, \\[6pt]
\sqrt{ \left( x + \frac{1}{2} \right)^{2} + \left( y + \frac{1}{2} \right)^{2} }, 
& -1 \leq x \leq 0,\; -1 \leq y \leq 0.
\end{cases}
\]
Let $\gamma^{\ell}=0.01,\ z^1=1,\ z^2=-1,\ \varepsilon_{\text{min}}=2\times10^{-4},\ \varepsilon_{\text{max}}=1.56\times 10^{-2}$ and the initial ion concentration is:
\[c^{\ell}(x,y,0)=0.1,~ \ell=1,2.\]


\bibliographystyle{plain}
\bibliography{refbib}

@article{BSD:PRE:2009,
  title                    = {Simple and robust solver for the {Poisson-Boltzmann} equation},
  author                   = {M. Baptista and R. Schmitz and B. D{\"u}nweg},
  journal                  = {Phys. Rev. E},
  year                     = {2009},
  pages                    = {016705},
  volume                   = {80}
}

@article{li2024finite,
  title={Finite-Difference Approximations and Local Algorithm for the {Poisson} and {Poisson-Boltzmann} Electrostatics},
  author={B. Li and Q. Yin and S. Zhou},
  journal={arXiv preprint arXiv:2409.15796},
  year={2024}
}

@article{CKC:BJ:03,
  title                    = {Dielectric self-energy in {Poisson-Boltzmann} and {Poisson-Nernst-Planck} models of ion channels},
  author                   = {Corry, B. and Kuyucak, S. and Chung, S.},
  journal                  = {Biophys. J.},
  year                     = {2003},
  Pages                    = {3594-3606},
  Volume                   = {84}
}

@article{daiguji:NL:2005,
  title                    = {Nanofluidic diode and bipolar transistor},
  author                   = {Daiguji, H. and Oka, Y. and Shirono, K.},
  journal                  = {Nano Lett.},
  year                     = {2005},
  pages                    = {2274--2280},
  volume                   = {5}
}

@article{duhamel1990fast,
  title={Fast {{F}}ourier transforms: a tutorial review and a state of the art},
  author={Duhamel, P. and Vetterli, M.},
  journal={Signal Process},
  volume={19},
  pages={259--299},
  year={1990}
}

@article{fahrenberger2014simulation,
  title                    = {Simulation of Electric Double Layers around Charged Colloids in Aqueous Solution of Variable Permittivity},
  author                   = {Fahrenberger, F. and Xu, Z. and Holm, C.},
  journal                  = {J. Chem. Phys.},
  year                     = {2014},
  Pages                    = {064902},
  Volume                   = {141},
}

@article{Gouy:JP:1910,
  title                    = {Constitution of the electric charge at the surface of an electrolyte},
  author                   = {G. Gouy},
  journal                  = {J. Phys.},
  year                     = {1910},
  pages                    = {457-468},
  volume                   = {9}
}

@article{honig1995classical,
  title                    = {Classical electrostatics in biology and chemistry},
  author                   = {B. Honig and A. Nicholls},
  journal                  = {Science},
  year                     = {1995},
  pages                    = {1144-1149},
  volume                   = {268}
}

@article{jiang2018improved,
  title={Improved local lattice {M}onte {C}arlo simulation for charged systems},
  author={Jiang, J. and Wang, Z.},
  journal={J. Chem. Phys.},
  volume={148},
  pages={114105},
  year={2018}
}

@article{kim2005fully,
  title={Fully implicit particle-in-cell algorithm.},
  author={Kim, H. and Chacon, L. and Lapenta, G.},
  journal={Bull. Am. Phys. Soc.},
  year={2005}
}

@article{kirsch2015mathematical,
  title={The mathematical theory of {T}ime-{H}armonic {M}axwell’s equations},
  author={Kirsch, A. and Hettlich, F.},
  journal={Appl. Math. Sci.},
 pages={20},
  volume={190},
  year={2015}
}

@article{levrel2005monte,
  title={Monte {{C}}arlo algorithms for charged lattice gases},
  author={Levrel, L. and Maggs, A. C.},
  journal={Phys. Rev. E},
  volume={72},
  pages={016715},
  year={2005}
}

@article{M:JCP:2002,
  title={Dynamics of a local algorithm for simulating {C}oulomb interactions},
  author={Maggs, A. C.},
  journal={J. Chem. Phys.},
  volume={117},
  pages={1975--1981},
  year={2002}
}

@book{mccormick1987multigrid,
  title={Multigrid methods},
  author={McCormick, S. F.},
  year={1987},
  publisher={SIAM}
}

@article{mesa1996image,
  title={Image charge method for electrostatic calculations in field-emission diodes},
  author={Mesa, G. and Dobado-Fuentes, E. and S\'{a}enz, J. J.},
  journal={J. Appl. Phys.},
  volume={79},
  pages={39--44},
  year={1996}
}

@article{liang2021high,
  title={A high-accurate fast Poisson solver based on harmonic surface mapping algorithm},
  author={Liang, J. and Liu, P. and Xu, Z.},
  journal={Commun. Comput. Phys.},
  year={2021},
  volume={30},
  pages={210-226}
}

@article{milaszewicz1987improving,
  title={Improving {J}acobi and {G}auss-{S}eidel iterations},
  author={Milaszewicz, J.},
  journal={Linear Algebra Appl.},
  volume={93},
  pages={161--170},
  year={1987}
}

@book{monk2003finite,
  title={Finite element methods for Maxwell's equations},
  author={Monk, P.},
  year={2003},
  publisher={Oxford university press}
}

@book{nussbaumer1982fast,
  title={The fast Fourier transform},
  author={Nussbaumer, H. J. },
  year={1982},
  publisher={Springer}
}

@article{pasichnyk2004coulomb,
  title                    = {Coulomb interactions via local dynamics: {A} molecular-dynamics algorithm},
  author                   = {Pasichnyk, I. and D{\"u}nweg, B.},
  journal                  = {J. Phys.: Condens. Matter},
  year                     = {2004},
  pages                    = {S3999},
  volume                   = {16}
}

@article{qiao2023maxwell,
  title={A {M}axwell--{A}mp{\`e}re {N}ernst--{P}lanck framework for modeling charge dynamics},
  author={Qiao, Z. and Xu, Z. and Yin, Q. and Zhou, S.},
  journal={SIAM J. Appl. Math.},
  volume={83},
  pages={374--393},
  year={2023},
  publisher={SIAM}
}

@article{qiao2023structure,
  title={Structure-preserving numerical method for Maxwell-Amp{\`e}re Nernst-Planck model},
  author={Qiao, Z. and Xu, Z. and Yin, Q. and Zhou, S.},
  journal={J. Comput. Phys.},
  volume={475},
  pages={111845},
  year={2023},
  publisher={Elsevier}
}

@article{qiao2025intrinsic,
  title={Intrinsic local {G}auss's law preserving {PIC} method: A self-consistent field-particle update scheme for plasma simulations},
  author={Qiao, Z. and Xu, Z. and Yin, Q. and Zhou, S.},
  journal={arXiv preprint arXiv:2506.02407},
  year={2025}
}

@article{rottler2004continuum,
  title={A continuum, {O (N)} {M}onte {C}arlo algorithm for charged particles},
  author={Rottler, J. and Maggs, A. C.},
  journal={J. Chem. Phys.},
  volume={120},
  pages={3119--3129},
  year={2004}
}

@article{roux2002theoretical,
  title={Theoretical and computational models of ion channels},
  author={Roux, B.},
  journal={Curr. Opin. Struct. Biol.},
  volume={12},
  pages={182--189},
  year={2002}
}

@incollection{ruge1987algebraic,
  title={Algebraic multigrid},
  author={Ruge, J. W. and St{\"u}ben, K.},
  booktitle={Multigrid methods},
  pages={73--130},
  year={1987},
  publisher={SIAM}
}

@article{stuben2001review,
  title={A review of algebraic multigrid},
  author={St{\"u}ben, K.},
  journal={Numerical Analysis: Historical Developments in the 20th Century},
  pages={331--359},
  year={2001}
}

@article{waisman1972mean,
  title={Mean spherical model integral equation for charged hard spheres I. Method of solution},
  author={Waisman, E. and Lebowitz, J. L.},
  journal={J. Chem. Phys.},
  volume={56},
  pages={3086--3093},
  year={1972}
}

@article{hemker1981introduction,
  title={Introduction to multigrid methods},
  author={Hemker, P. W.},
  journal={Nieuw Arch. Wisk.},
  volume={3},
  pages={71--101},
  year={1981}
}

@article{zheng2011second,
  title={Second-order {P}oisson--{N}ernst--{P}lanck solver for ion transport},
  author={Zheng, Q. and Chen, D. and Wei, G. W.},
  journal={J. Comput. Phys.},
  volume={230},
  pages={5239--5262},
  year={2011}
}

@article{ZWL:PRE:2011,
  title                    = {Mean-field description of ionic size effects with non-uniform ionic sizes: {A} numerical approach},
  author                   = {S. Zhou and Z. Wang and B. Li},
  journal                  = {Phys. Rev. E},
  year                     = {2011},
  pages                    = {021901},
  volume                   = {84}
}

@article{maggs2002local,
  title={Local simulation algorithms for {C}oulomb interactions},
  author={Maggs, A. C. and Rossetto, V.},
  journal={Phys. Rev. Lett.},
  volume={88},
  pages={196402},
  year={2002}
}

@article{izmailov2003karush,
  title={Karush-{K}uhn-{T}ucker systems: regularity conditions, error bounds and a class of {N}ewton-type methods},
  author={Izmailov, A. F. and Solodov, M.},
  journal={Mathematical Programming},
  volume={95},
  pages={631--650},
  year={2003}
}

@book{zienkiewicz1977finite,
  title={The finite element method},
  author={Zienkiewicz, O. C. and Taylor, R. L. and Nithiarasu, P. and Zhu, J. Z.},
  year={1977},
  publisher={Elsevier}
}

@book{dhatt2012finite,
  title={Finite element method},
  author={Dhatt, G. and Lefran{\c{c}}ois, E. and Touzot, G.},
  year={2012},
  publisher={John Wiley \& Sons}
}

@article{cheng2005heritage,
  title={Heritage and early history of the boundary element method},
  author={Cheng, A. H. D. and Cheng, D. T.},
  journal={Eng. Anal. Bound. Elem.},
  volume={29}, 
  pages={268--302},
  year={2005}
}

@incollection{hall1994boundary,
  title={Boundary element method},
  author={Hall, W. S.},
  booktitle={The boundary element method},
  pages={61--83},
  year={1994},
  publisher={Springer}
}

@article{van2006charge,
  title={Charge inversion at high ionic strength studied by streaming currents},
  author={van der Heyden, F. H. and Stein, D. and Besteman, K. and Lemay, S. G. and Dekker, C.},
  journal={Phys. Rev. Lett.},
  volume={96},
  pages={224502},
  year={2006}
}

@book{hockney2021computer,
  title={Computer simulation using particles},
  author={Hockney, R. W. and Eastwood, J. W.},
  year={2021},
  publisher={CRC Press}
}

@article{clapham2007calcium,
  title={Calcium signaling},
  author={Clapham, D. E.},
  journal={Cell},
  volume={131},
  pages={1047--1058},
  year={2007}
}

@article{xu2014solving,
  title={Solving fluctuation-enhanced {P}oisson--{B}oltzmann equations},
  author={Xu, Z. and Maggs, A. C.},
  journal={J. Comput. Phys.},
  volume={275},
  pages={310--322},
  year={2014}
}

@article{tu2022linear,
  title={Linear-scaling selected inversion based on hierarchical interpolative factorization for self {G}reen's function for modified {P}oisson-{B}oltzmann equation in two dimensions},
  author={Tu, Y. and Pang, Q. and Yang, H. and Xu, Z.},
  journal={J. Comput. Phys.},
  volume={461},
  pages={110893},
  year={2022}
}

@article{ma2021modified,
  title={Modified {P}oisson--{N}ernst--{P}lanck model with {C}oulomb and hard-sphere correlations},
  author={Ma, M. and Xu, Z. and Zhang, L.},
  journal={SIAM J. Appl. Math.},
  volume={81},
  pages={1645--1667},
  year={2021}
}

@article{eisenberg2007poisson,
  title={Poisson--{N}ernst--{P}lanck systems for ion channels with permanent charges},
  author={Eisenberg, B. and Liu, W.},
  journal={SIAM J. Math. Anal.},
  volume={38},
  pages={1932--1966},
  year={2007}
}

@article{maury2001fat,
  title={A fat boundary method for the {P}oisson problem in a domain with holes},
  author={Maury, B.},
  journal={J. Sci. Comput.},
  volume={16},
  pages={319--339},
  year={2001}
}

@article{sorokina2022interpolated,
  title={An interpolated {G}alerkin finite element method for the {P}oisson equation},
  author={Sorokina, T. and Zhang, S.},
  journal={J. Sci. Comput.},
  volume={92},
  pages={47},
  year={2022}
}

@article{desiderio2022cvem,
  title={{CVEM-BEM} coupling with decoupled orders for 2D exterior {P}oisson problems},
  author={Desiderio, L. and Falletta, S. and Ferrari, M. and Scuderi, L.},
  journal={J. Sci. Comput.},
  volume={92},
  pages={96},
  year={2022}
}

@article{feng2021fft,
  title={{FFT}-based high order central difference schemes for {P}oisson’s equation with staggered boundaries},
  author={Feng, H. and Long, G. and Zhao, S.},
  journal={J. Sci. Comput.},
  volume={86},
  pages={7},
  year={2021}
}

\end{document}